%% file: main.tex
\documentclass[reqno,twoside]{amsart}
\usepackage{amsfonts}
\usepackage{amsmath,amssymb}
\usepackage{epic}
\usepackage{pstricks}
\usepackage{graphicx}
\usepackage{xcolor}
\usepackage{upgreek}
\usepackage[ruled,vlined,linesnumbered]{algorithm2e}
\usepackage{enumitem}
\usepackage{algorithmic}
\usepackage{booktabs}
\usepackage{subcaption}
\usepackage[colorlinks]{hyperref} 
\usepackage[
  backend=biber,
  style=numeric-comp,
  doi=true,
  url=false,
  giveninits=true
]{biblatex}
\AtBeginBibliography{%
  \DeclareNameAlias{author}{family-given}
  \DeclareNameAlias{editor}{family-given}
}
\usepackage{tikz}
\usetikzlibrary{arrows.meta, positioning}

\usepackage{a4wide,amsmath,amssymb,latexsym,amsthm}
\usepackage{eucal}
\usepackage{graphicx}

\newtheorem{theorem}{Theorem}[]
\newtheorem{proposition}[]{Proposition}
\newtheorem{remark}[]{Remark}
\newtheorem{lemma}[]{Lemma}
 
\newtheorem{definition}[]{Definition}
\newtheorem{example}[]{Example}

\newcommand{\be}{\begin{equation}}
\newcommand{\ee}{\end{equation}}
\newcommand{\ba}{\begin{eqnarray}}
\newcommand{\ea}{\end{eqnarray}}
\newcommand{\beq}{\begin{equation}}
\newcommand{\eeq}{\end{equation}}

\usepackage{color}

\usepackage[textsize=small]{todonotes}

\numberwithin{equation}{section}

\usepackage{color}
\usepackage{stmaryrd}
\usepackage{comment}

\subjclass[2010]{}

\begin{document}
\title[Neural Network Dual Norms for Minimal Residual Finite Element Methods]{Neural Network Dual Norms for Minimal Residual Finite Element Methods}

\author{Hamd Alsobhi}
\address{H. Alsobhi,  School of Mathematical Sciences, University of Nottingham, UK and
Department of Mathematics, Islamic University of Madinah, Saudi Arabia}
\email{Hamd.alsobhi@nottingham.ac.uk, Hamd.alsobhi@iu.edu.sa}

\author{Emin Benny-Chacko}
\address{E. Benny-Chacko,  School of Mathematical Sciences, University of Nottingham, UK}
\email{Emin.BennyChacko1@nottingham.ac.uk}

\author{Ignacio Brevis}
\address{I. Brevis,  School of Mathematical Sciences, University of Nottingham, UK}
\email{ignacio.brevis@nottingham.ac.uk, ignacio.brevis.v@gmail.com}

\author{Kristoffer G. van der Zee}
\address{K.~G. van der Zee,  School of Mathematical Sciences, University of Nottingham, UK}
\email{KG.vanderZee@nottingham.ac.uk}

\begin{abstract}
Minimal-residual methods for PDEs with a residual in a dual space are non-trivial to guarantee stability.
We present a minimal-residual finite element method in which the solution space is a standard finite element space, but neural networks are used as test functions for the evaluation of residual dual norms. The use of a  neural network improves the approximation of the residual representer, and thereby improves the stability of the method.
Our hybrid approach is implemented through a deep residual Uzawa algorithm that alternates finite element updates with neural network training. We prove consistency and convergence results for the Uzawa methodology. We also prove an a~priori error estimate that relies on a suitable Fortin compatibility condition. 
Numerical experiments on 
advection-reaction problems with singular or discontinuous data show that the proposed framework delivers robust and accurate approximations. 

\vspace{5pt}        

\vspace{5pt}  

\end{abstract}
\maketitle
\tableofcontents

\input{sections/section1_new3}

\input{sections/LR}

\input{sections/section2}

\input{sections/section3}

\input{sections/section4}
\input{sections/section5}

\input{sections/section7}

\section{Acknowledgments}
KvdZ is grateful to Luis Espath and Ignacio Muga for discussions on the topic.
The work done by HA was supported financially by the Islamic University of Madinah and the Ministry of Education of Saudi Arabia. EBC is supported by the European Union’s Horizon Europe research and innovation programme under the Marie Sklodowska-Curie grant agreement No 101119556.
The research of IB and KvdZ was supported by the Engineering and Physical Sciences Research Council (EPSRC), UK, under Grant EP/W010011/1.
\clearpage
\newpage

\printbibliography

\end{document}

%% file: sections/section1_new3.tex
\section{Introduction}

Solving partial differential equations (PDEs) remains a significant challenge in computational science. Numerical methods for these problems can be unstable, making it difficult to achieve accurate and robust solutions. 
For example, classical finite element methods (FEM), particularly in their Galerkin formulation, are widely used for approximating PDEs~\cite{ brenner2008mathematical, ciarlet2002finite, johnson2009numerical, larson2013finite}. However, they may exhibit instabilities depending on the nature of the underlying equations~\cite{  brooks1982streamline, burman2010consistent, codina2000stabilization}. To address these issues, minimal residual (MinRes) finite element methods, which are developed mainly in the discontinuous Petrov-Galerkin (DPG) context as a natural generalization of least squares methods~\cite{demkowicz2010class, demkowicz2011class, 
zitelli2011class}, provide a more stable alternative. By minimizing the residual in a suitable norm, MinRes methods improve numerical robustness and offer a reliable framework for error control~\cite{gopalakrishnan2014analysis, muga2020discretization}. Their variational structure also facilitates the derivation of a posteriori error estimates, as the residual is computed directly in the process~\cite{ carstensen2014posteriori, gopalakrishnan2013five}.

Despite their theoretical strengths, MinRes methods face practical limitations. A key challenge in the formulation of MinRes methods lies in the nature of the dual norm minimization. Specifically, solving a minimal residual problem requires minimizing the residual in the dual norm of the test space~\cite{carstensen2014posteriori, gopalakrishnan2013five, muga2020discretization}. This, by definition, involves taking a supremum over an infinite-dimensional space~\cite{ern2004theory}. For a given discrete trial space \( U_h \subset U \), the MinRes problem takes the form
\begin{equation}\label{minres:full_dual_norm}
    u_h := \arg \min_{u_h \in U_h} \| f - B u_h \|_{V^*}
    = \arg \min_{u_h \in U_h} \sup_{v \in V} \frac{\langle f - B u_h, v \rangle}{\|v\|_V},
\end{equation}
where \( B : U \rightarrow V^* \) is the operator associated with the variational formulation, and \( f \in V^* \) is the right-hand side. The supremum in the dual norm is taken over the continuous infinite-dimensional test space \( V \) (e.g., \( V = H^1 \))~\cite{ broersen2018stability, demkowicz2023mathematical, muga2020discretization}. 

To make the problem computationally tractable, one can approximate the supremum by constructing a finite-dimensional subspace \( V_h \subset V \). However, not every subspace yields a stable or accurate approximation~\cite{zitelli2011class, carstensen2014posteriori, muga2020discretization}. The discrete test space must be carefully designed to satisfy key conditions i.e.,  the inf-sup stability or Fortin condition. Therefore, designing an appropriate discrete test space is a central challenge in MinRes methods. 
These limitations become significant when applied to certain classes of PDEs. Their stability can be difficult to guarantee, particularly in convection-dominated or singularly perturbed regimes~\cite{roos2012robust, fuhrer2024robust}.

Recent advances in machine learning, particularly in deep neural networks, have introduced new opportunities for addressing these challenges. Neural networks are capable of approximating complex, non-smooth functions. They have shown strong performance in learning sharp interfaces, discontinuities, and multiscale structures~\cite{higham2019deep, karniadakis2021physics, peng2021multiscale, uriarte2023deep, weinan2020machine, ainsworth2021galerkin}. These capabilities have motivated the development of hybrid numerical methods. These hybrid methods may combine classical finite element discretisations with neural network approximations of challenging components such as the residual's dual norm. Notably, variational deep learning methods, such as the Deep Ritz method proposed by Weinan E and Bing Yu~\cite{yu2018deep}, have successfully integrated neural networks within variational frameworks. These approaches show promise for high-dimensional and complex PDE problems, see e.g.,~\cite{brevis2021machine,cai2020deep, cai2024efficient, kumar2025deep}.

Building upon these insights, the aim of this paper is to introduce a novel minimal residual finite element framework in which the dual norm is approximated using neural networks. The approach employs a standard finite element space \( U_h \) for the trial functions and replaces the conventional finite-dimensional test space with a neural network set \( \mathcal{M} \), which is a class of neural network functions. Specifically, the dual norm of the residual is computed by taking the supremum over \( \mathcal{M} \). 
This leads to the modified minimal residual formulation:
\begin{equation}\label{minres:nn_dual_norm}
    u_h := \arg \min_{u_h \in U_h} \sup_{v \in \mathcal{M}} \frac{\langle f - B u_h, v \rangle}{\|v\|_{V}}.
\end{equation}
In this setting, the neural network is trained to solve the supremum problem and thus acts as a residual representative, capturing the direction in which the residual is maximized. This formulation enables a flexible and adaptive test space approximation without requiring an explicitly constructed basis~\cite{brevis2022neural, liu2025dual, rojas2024robust}.

The main contributions of this study can be summarized in three points. First, using an abstract MinRes 
formulation, we propose a solution
method for~\eqref{minres:nn_dual_norm}, which employs an Uzawa-type iterative algorithm~\cite{benzi2005numerical, uzawa1958iterative}, alternating between finite element updates of the solution and neural network updates of the residual representation. We demonstrate consistency of our Uzawa problem formulation with~\eqref{minres:nn_dual_norm}. Second, we provide a rigorous theoretical analysis of the proposed method. This includes both a priori and a posteriori error estimates, thereby establishing its stability and convergence, as well as conditions under which the Uzawa method is expected to converge. The error estimates rely on a Fortin condition similar to the one proposed by Rojas et al.~\cite{rojas2024robust}. Our proof of Uzawa convergence is inspired by the unified analysis in Bacuta~\cite{bacuta2006unified}, which provides a rigorous abstract framework for establishing convergence of Uzawa-type algorithms. Third, we demonstrate through numerical experiments that the proposed method achieves robust and accurate approximations for both the solution and the residual. The observed convergence behavior is consistent with the theoretical error estimates. These results confirm that the neural-network-based set can provide a powerful test space construction.



The remainder of the paper is structured as follows. Section~\ref{sec2:chapter7} presents the abstract variational formulation and introduces the hybrid computational methodology. Section~\ref{sec3:chapter7} develops the theoretical analysis, including well-posedness and error bounds. Section~\ref{sec4:chapter7} details the implementation for weak advection-reaction problems. It includes the model formulation, the design of a one-dimensional shallow ReLU neural network, and the construction of the deep residual Uzawa algorithm. Section~\ref{sec5:chapter7} addresses implementation aspects and reports numerical results for representative test cases. Section~\ref{sec:conclusions} concludes with a summary of findings and a discussion of potential future research directions.

%% file: sections/LR.tex
\section{Literature Review}

The formulation of MinRes finite element methods has a long and rich history in the numerical solution of PDEs. This is especially true for problems where stability is challenging, such as convection-dominated or first-order systems. These methods aim to minimize the residual in a suitable dual norm, providing a principled and robust approach to discretization.

A significant milestone in this direction is the DPG methodology, originally developed and formalized by Demkowicz and Gopalakrishnan, see, e.g.,~\cite{demkowicz2010class, demkowicz2011class, demkowicz2017discontinuous, demkowicz2012class}. In DPG methods, the optimal test functions are computed locally to approximate the dual norm of the residual. This strategy ensures stability and naturally yields a posteriori error estimators. These methods can be viewed as concrete realizations of the MinRes framework and have been applied successfully to a variety of problems, from diffusion to convection-dominated transport~\cite{demkowicz2023mathematical}.

Recent developments have focused on the construction of test spaces in settings where classical analytical techniques are either infeasible or overly complex. For instance, Führer and Heuer~\cite{fuhrer2024robust} constructed robust test spaces for DPG in singularly perturbed regimes using advanced analytical tools. In contrast, our approach uses neural networks to represent test functions implicitly by solving the supremum problem associated with the dual norm of the residual, leading to a flexible and data-driven alternative to classical constructions.

Beyond Hilbert space settings, researchers have also extended residual minimization methods to Banach spaces. Muga et al.~\cite{muga2020discretization} introduced dual Petrov-Galerkin schemes in Banach spaces to accommodate broader norm structures. Likewise, Cai and colleagues~\cite{cai2024least} developed least-squares and MinRes strategies in negative Sobolev spaces for first-order systems, enhancing stability in complex applications.

Machine learning-based formulations have recently emerged as powerful tools for approximating test spaces or residual representations. Brevis et al.~\cite{brevis2021machine, brevis2022neural,Brevis2024} proposed the machine-learning minimal-residual (ML-MRes) framework, where a neural network parametrizes the test space inner product to achieve goal-oriented discretization. Meanwhile, robust variational physics-informed neural networks (RVPINNs)~\cite{rojas2024robust} introduce a loss function that approximates the dual norm of the residual directly, bringing neural solvers conceptually closer to MinRes FEM.

While traditional strategies such as DPG~\cite{demkowicz2017discontinuous}, FOSLS~\cite{cai1994first, cai1997first, berndt2005analysis, berndt2006analysis}
, and goal-oriented Petrov-Galerkin methods~\cite{ellis2014locally} have laid the theoretical foundation, modern neural-network-based formulations aim to retain their benefits while offering greater scalability and adaptability. The present work builds directly on this trend by introducing a neural-network-enhanced MinRes finite element framework capable of capturing dual norm behavior through learned test functions.

%% file: sections/section2.tex
\section{Abstract Setting and Methodology}
\label{sec2:chapter7}

In this section, we begin with the abstract variational setting and then the corresponding saddle-point system. Next, the Uzawa algorithm is used to solve the coupled problem, with residual minimization performed via a Deep Ritz framework.

\subsection{Abstract Setting}

Let $U$ and $V$ be Hilbert spaces with respective duals $U^\ast$ and $V^\ast$. We consider a continuous linear operator $B : U \to V^\ast$ and a bounded linear functional $f \in V^\ast$. The abstract variational problem is to find $u \in U$ such that
\begin{equation}
\label{eq:abs_var_form}
b(u, v) = f(v), \qquad \forall\, v \in V,
\end{equation}
where the bilinear form is given by $b(u,v) := \langle B u, v \rangle_{V^\ast, V}$. Here, $\langle \,\cdot\, ,  \cdot\, \rangle_{V^\ast, V}$ denotes the duality pairing between $V^\ast$ and $V$.

To approximate the solution, we minimize the dual norm of the residual over a finite-dimensional subspace $U_h \subset U$. The associated least-squares functional is defined as:
\begin{equation}
J(w) := \frac{1}{2} \|B w - f\|^2_{V^\ast}, \qquad \forall\, w \in U.
\end{equation}
The minimal-residual approximation is then obtained by solving:
\begin{equation}
\label{eq:semi_inf_MinRes}
u_h = \arg \min_{w_h \in U_h} J(w_h).
\end{equation}

As discussed in~\cite{cohen2012adaptivity}, this minimization problem is equivalent to a mixed (saddle-point) formulation. Specifically, we seek $(u_h, r) \in U_h \times V$ such that
\begin{equation}\label{saddlproblem7}
\begin{aligned}
(r, v)_{V} + b(u_h, v) &= f(v), \qquad \forall\, v \in V, \\
b(w_h, r) &= 0, \qquad \forall\, w_h \in U_h.
\end{aligned}
\end{equation}


\subsubsection*{Well-Posedness: Inf--Sup Condition}
To ensure that the abstract problem~\eqref{eq:abs_var_form} is well-posed, we use the classical inf-sup (or Babu\v{s}ka--Brezzi) condition, 
here stated in the form of the Banach--Ne\v{c}as--Babu\v{s}ka theorem 
(see, e.g., \cite[Theorem~4.2.3]{ern2004theory},
\cite[Theorem~2.6]{boffi2013mixed}, or \cite[Theorem~2.5]{di2011mathematical}).
There exists a constant $\mu > 0$ such that
\begin{equation}
\label{eq:inf-sup}
\inf_{u \in U \setminus \{0\}} \; \sup_{v \in V \setminus \{0\}}
\frac{|b(u, v)|}{\|u\|_{U} \, \|v\|_{V}} \; \ge \; \mu > 0,
\end{equation}
and
\begin{equation}
\label{eq:injectivity}
\forall\, v \in V, \quad \big( b(u, v) = 0 \ \ \forall\, u \in U \big) \ \Rightarrow \ v = 0.
\end{equation}
Under these conditions, the solution $u \in U$ of~\eqref{eq:abs_var_form} 
is unique and satisfies the stability estimate.
\begin{equation}
\label{eq:stability}
\|u\|_{U} \; \le \; \frac{C_b}{\mu} \, \|f\|_{V^*},
\end{equation}
where $C_b > 0$ denotes the continuity constant of the bilinear form 
$b(\cdot,\cdot)$.
\subsection{Uzawa Iteration}

To solve the saddle-point system iteratively, we use a variant of the classical Uzawa algorithm. Given a current approximation \( u_h^k \in U_h \), we compute the residual \( r \in V \) by solving:
\begin{equation}
(r, v)_{V} + \langle B u_h^k, v \rangle = \langle f, v \rangle, \qquad \forall\, v \in V.
\end{equation}

This step can also be interpreted as the minimization problem:
\begin{equation}
r = \arg \min_{r \in V} \left( \frac{1}{2} \|r\|^2_{V} - \langle f - B u_h^k, r \rangle \right),
\end{equation}
which reflects a least-squares characterization of the residual in the test space norm.
After obtaining \( r \), the next approximation of the solution \(u\) is obtained by a relaxed update:
\begin{equation}
\langle u_h^{k+1}, w_h \rangle = \langle u_h^k, w_h \rangle + \rho\, b(w_h, r), \qquad \forall\, w_h \in U_h,
\end{equation}
where \( \rho > 0 \) is a relaxation parameter.

This results in the following iterative scheme:
\begin{equation}\label{eq:uzawa-res}
\left\{
\begin{aligned}
&\text{(1) Solve for } r^{k+1} \in V: \\
&\qquad (r^{k+1}, v)_{V} = \langle f, v \rangle - b(u_h^k, v), \quad \forall\, v \in V, \\[0.5em]
&\text{(2) Update } u_h^{k+1} \in U_h: \\
&\qquad \langle u_h^{k+1}, w_h \rangle = \langle u_h^k, w_h \rangle + \rho\, b(w_h, r^{k+1}), \quad \forall\, w_h \in U_h.
\end{aligned}
\right.
\end{equation}

We now discuss the neural network approximation of the residual.


\subsection{Deep Ritz-Based Residual Minimization}

To improve adaptability in the residual computation, we introduce a neural network parameterization given by the set  \( \mathcal{M}_n \subset V := H^1(\Omega) \) to approximate the test space $V$. This idea is motivated by the Deep Ritz method~\cite{yu2018deep}, which frames variational problems as neural network optimization tasks.

In this setting, the residual update step becomes:
\begin{equation}
\label{eq:hybrid-minres}
\left\{
\begin{aligned}
&\text{Given } u_h^k \in U_h, \text{ find } r_n \in \mathcal{M}_n \text{ such that:} \\
&\qquad r_n = \arg\min_{v_n \in \mathcal{M}_n} \left( \frac{1}{2} \|v_n\|^2_{V} - \langle f - B u_h^k, v_n \rangle \right), \\[0.5em]
&\text{then update } u_h^{k+1} \in U_h \text{ by:} \\
&\qquad \langle u_h^{k+1}, w_h \rangle = \langle u_h^k, w_h \rangle + \rho\, b(w_h, r_n), \quad \forall\, w_h \in U_h.
\end{aligned}
\right.
\end{equation}

Here, the residual function \( r_n \) is approximated using a neural network with parameters \( \theta \), i.e., \( r_n(x; \theta) \in \mathcal{M}_n \). The network is trained by minimizing the functional:
\begin{equation}
\mathcal{J}_r(\theta) = \frac{1}{2} \|r_n(\cdot; \theta)\|^2_{V} - \langle f - B u_h^k, r_n(\cdot; \theta) \rangle.
\end{equation}

Upon the convergence of Uzawa iteration, $u_h$ is the solution of the following MinRes formulation with neural network dual norm: 
%
\begin{equation}\label{dis: full dual norm}
    u_h := \arg \min_{} \| f- B u_h \|_{\mathcal{M}_n^*}
     = \arg \min \sup_{v_n \in \mathcal{M}_n} \frac{\langle f- B u_h, v_n \rangle}{\|v_n\|_{V}}.
\end{equation}

The MinRes formulation with a neural network dual norm \eqref{dis: full dual norm} relies critically on how the residual is represented and updated. To make this connection precise, we first establish an equivalence between the supremum and minimization formulations of the residual. We then use this characterization to study the convergence properties of the inexact Uzawa iteration.

\subsection{Residual Formulation and Inexact Uzawa Convergence}
\label{sec:residual-uzawa}
In particular, for a fixed \( u_h \in U_h \), we define the residual \( r_n \in \mathcal{M}_n \) to be the Riesz representative of the functional \( f - B u_h \in \mathcal{M}_n^* \). 
This is initially expressed in supremum form, but can be equivalently reformulated as a minimization. We now formalize this equivalence:

\begin{lemma}[Equivalence of Supremum and Minimization Formulations]
\label{lemma:sup-min-equivalence-direct}
Let \( f \in V^* \), \( u_h \in U_h \), and \( B : U \to V^* \) a bounded linear operator. Suppose \( \mathcal{M}_n \) is a finite-dimensional residual test set
(e.g., a set of shallow ReLU network functions). 
Then the optimal residual \( r_n^* \in \mathcal{M}_n \) defined by
\begin{equation}
\label{eq:supremum-form}
r_n^* = \arg\sup_{v_n \in \mathcal{M}_n} \frac{\langle f - B u_h, v_n \rangle}{\|v_n\|}
\end{equation}
subject to the normalization \( \|r_n^* \|^2 = \langle f - B u_h , r_n^* \rangle ,\) is the solution of the unconstrained minimization problem:
\begin{equation}
\label{eq:minimisation-form}
r_n^* = \arg\min_{v_n \in \mathcal{M}_n} \left( \frac{1}{2} \|v_n\|^2 - \langle f - B u_h, v_n \rangle \right).
\end{equation}
\end{lemma}

\begin{proof}
We begin with the supremum formulation:
\[
r_n^* = \arg\sup_{v_n \in \mathcal{M}_n} \frac{\langle f - B u_h, v_n \rangle}{\|v_n\|},
\]
subject to the constraint
\[
\|v_n\|^2 = \langle f - B u_h, v_n \rangle.
\]

Under this constraint, we observe that
\[
\frac{\langle f - B u_h, v_n \rangle}{\|v_n\|} = \sqrt{\langle f - B u_h, v_n \rangle},
\]
so maximizing the quotient is equivalent to maximizing the inner product:
\[
\arg\sup_{\substack{v_n \in \mathcal{M}_n \\ \|v_n\|^2 = \langle f - B u_h, v_n \rangle}} \langle f - B u_h, v_n \rangle,
\]
which is equivalent to
\[
\arg\inf_{\substack{v_n \in \mathcal{M}_n \\ \|v_n\|^2 = \langle f - B u_h, v_n \rangle}} -\langle f - B u_h, v_n \rangle.
\]

To derive an unconstrained problem, we construct an equivalent objective whose minimum is achieved at the same point. Since on the constraint set we have \( \langle f - B u_h, v_n \rangle = \|v_n\|^2 \), we observe that
\[
-\langle f - B u_h, v_n \rangle = -\|v_n\|^2.
\]
This suggests the function
\[
J(v_n) := \|v_n\|^2 - 2 \langle f - B u_h, v_n \rangle
\]
attains its minimum precisely when the constraint is satisfied, i.e., when \( \|v_n\|^2 = \langle f - B u_h, v_n \rangle \).

Therefore, minimizing \( J(v_n) \) over all \( \mathcal{M}_n \) recovers the same point as the constrained problem. Dividing by 2 (which does not change the minimizer), we obtain the equivalent unconstrained problem:
\[
r_n^* = \arg\min_{v_n \in \mathcal{M}_n} \left( \frac{1}{2} \|v_n\|^2 - \langle f - B u_h, v_n \rangle \right),
\]
which completes the proof.
\end{proof}

This network-based test space allows for automatic adaptation to problem features, such as discontinuities or sharp gradients. It enhances the flexibility and performance of residual minimization, especially in cases where constructing suitable test spaces manually is difficult.

We now analyze the Uzawa iteration \eqref{eq:uzawa-res} in the case where the residual is computed inexactly, for example, by a neural network.

Recall the Riesz maps
$R_V:V\to V^*$, $R_{U_h}:U_h \to U_h^*$ and the operator $B_h:U_h\to V^*$ with adjoint $B_h^*:V\to U_h^*$.

\begin{theorem}[Convergence of the inexact Uzawa iteration]\label{Thm_uzaw_con}
Let $(u,r)\in U_h \times V$ be the exact saddle-point solution of \eqref{saddlproblem7}. Consider the inexact Uzawa iteration
\[
r^k = R_V^{-1}(l - B_h u^k), 
\qquad 
u^{k+1} = u^k + \rho R_U^{-1} B_h^\dagger r^k_\delta,
\]
where $r^k$ is replaced by an approximation $r^k_\delta$ satisfying
\begin{equation}\label{assump1}
    \|r^k_\delta - r^k\|_V \leq \delta \|r^k\|_V.
\end{equation}
Let $0 < \rho < \tfrac{2}{C_b^2}$ and assume
\begin{equation}\label{delta_assump}
    \delta < \frac{1 - \omega}{2 C_b^2},
    \qquad 
    \omega = \max \{\, |1 - \rho \mu^2|, \; |1 - \rho C_b^2| \,\},
\end{equation}
then the sequence $(u^k, r^k_\delta)$ converges to the exact solution $(u_h,r) \in U_h \times V$.
\end{theorem}

\begin{remark}
Theorem~\ref{Thm_uzaw_con} illustrates that if the neural network approximation 
of the residual $r_n$ in \eqref{eq:hybrid-minres} is sufficiently accurate, 
then the inexact Uzawa iteration converges to the exact solution.
\end{remark}

\begin{proof}
Let $e_r^k = r^k_\delta - r$ and $e_u^k = u_h^k - u_h$, where $(u_h,r)$ is the exact solution.  
From the definition of $r^k$ it follows that
\[
R_V r^k = l - B_h u_h^k,
\qquad 
u_h^{k+1} = u_h^k + \rho R_{U_h}^{-1} B_h^\dagger r^k_\delta.
\]
Subtracting the relations for $r^k$ and $r$ yields
\begin{equation}\label{update1}
    R_V (r^k - r) = -B_h(u_h^k - u_h) = -B_h e_u^k.
\end{equation}
Consequently,
\[
R_V (r_\delta^k - r) 
    = R_V( r_\delta^k - r^k ) - B_h e_u^k,
\]
and therefore
\[
R_V e_r^k = R_V( r_\delta^k - r^k) - B_h e_u^k.
\]
Taking norms gives
\begin{align*}
    \|e_r^k\|_V 
    &\leq \|B_h e_u^k\|_{V^*} + \| r_\delta^k - r^k\|_V \\
    &\leq C_b \|e_u^k\|_U + \delta \|r^k\|_V \quad \text{(by \eqref{assump1})} \\
    &\leq C_b \|e_u^k\|_U + \delta \|r^k - r\|_V \\
    &\leq C_b \|e_u^k\|_U + \delta C_b \|e_u^k\|_U \quad \text{(by \eqref{update1})},
\end{align*}
so that
\begin{equation}\label{er&eu}
    \|e_r^k\|_V \leq C_b(1 + \delta)\, \|e_u^k\|_U.
\end{equation}
Turning to the update for $u_h^k$, we obtain
\begin{align*}
    e_u^{k+1} 
    &= u_h^{k+1} - u_h 
     = u_h^k - u_h + \rho R_{U_h}^{-1} B_h^\dagger (r_\delta^k - r) \\
    &= e_u^k + \rho R_{U_h}^{-1} B_h^\dagger \big[(r_\delta^k - r^k) + (r^k - r)\big] \\
    &= e_u^k + \rho R_{U_h}^{-1} B_h^\dagger (r_\delta^k - r^k) 
       - \rho R_{U_h}^{-1} B_h^\dagger R_V^{-1} B_h (u_h^k - u_h) \quad \text{(by \eqref{update1})} \\
    &= \big(I - \rho R_{U_h}^{-1} B_h^\dagger R_V^{-1} B_h\big) e_u^k 
       + \rho R_{U_h}^{-1} B_h^\dagger (r_\delta^k - r^k).
\end{align*}
Hence
\begin{align*}
    \|e_u^{k+1}\|_U 
    &\leq \omega \|e_u^k\|_U + \rho C_b \|r_\delta^k - r^k\|_V \\
    &\leq \omega \|e_u^k\|_U + \rho C_b \delta \|r^k\|_V \\
    &\leq \omega \|e_u^k\|_U + \rho C_b^2 \delta \|e_u^k\|_U \quad \text{(by \eqref{update1})}.
\end{align*}
This gives the recurrence
\[
    \|e_u^{k+1}\|_U \leq \big(\omega + \rho C_b^2 \delta\big) \|e_u^k\|_U.
\]
By assumption \eqref{delta_assump}, the factor satisfies $\omega + \rho C_b^2 \delta < 1$, and thus $\|e_u^{k+1}\|_U < \|e_u^k\|_U$. It follows that $u_h^k \to u_h$, and from \eqref{er&eu} we also conclude that $r^k_\delta \to r$.
\end{proof}

%% file: sections/section3.tex
\section{Convergence Analysis}
\label{sec3:chapter7}

In this section, we analyze how a neural-network-parameterized set \( \mathcal{M}_n \subset H^1(\Omega) \) impacts the stability and accuracy of the approximation obtained via residual minimization. 

\subsection{Stability and Discrete Inf-Sup Condition}
To ensure discrete stability, we employ the concept of a Fortin operator.

\begin{definition}[Fortin Operator]\label{def: fortin operator}
Let \( U_h \subset U \) be a linear subspace and \( \mathcal{M}_n \subset V \) be a neural network subset. A linear operator \begin{equation}\label{fortin operator} \Pi : V \to \mathcal{M}_n \end{equation} is a Fortin operator if:
\begin{enumerate}
    \item \textbf{Stability:} \( \| \Pi v \|_{V} \leq C_\Pi \| v \|_{V} \), for all \( v \in V \),
    \item \textbf{Orthogonality:} \( b(u_h, v - \Pi v) = 0 \), for all \( u_h \in U_h \) and \( v \in V \).
\end{enumerate}
\end{definition}

\begin{proposition}[Discrete Inf-Sup Stability]
If the continuous inf-sup condition holds with constant \( \mu \), and a Fortin operator \( \Pi : V \to \mathcal{M}_n \) exists with stability constant \( C_\Pi \), then the discrete inf-sup condition is satisfied with
\begin{equation}
\label{eq:disc_infsup}
\inf_{u_h \in U_h \setminus \{0\}} \sup_{v_n \in \mathcal{M}_n \setminus \{0\}} \frac{b(u_h, v_n)}{\|u_h\|_{U} \|v_n\|_{V}} \geq \frac{\mu}{C_\Pi}.
\end{equation}
\end{proposition}

\begin{proof}
The proof follows the classical argument known as Fortin’s Lemma; see e.g.,~\cite[Section~5.4]{boffi2013mixed} and~\cite[Theorem~1]{ern2016converse}. Since the continuous inf-sup condition holds with constant \( \mu \), and a Fortin operator \( \Pi : V \to \mathcal{M}_n \) exists satisfying
\[
b(u_h, \Pi v) = b(u_h, v) \quad \text{for all } u_h \in U_h,\, v \in V,
\]
and
\[
\|\Pi v\|_V \leq C_\Pi \|v\|_V \quad \text{for all } v \in V,
\]
we proceed as follows. For any \( u_h \in U_h \setminus \{0\} \), the continuous inf-sup condition gives
\[
\sup_{v \in V \setminus \{0\}} \frac{b(u_h, v)}{\|v\|_V} \geq \mu \|u_h\|_U.
\]
Applying the Fortin property and the stability of \( \Pi \), we have
\[
\sup_{v_n \in \mathcal{M}_n \setminus \{0\}} \frac{b(u_h, v_n)}{\|v_n\|_V} \geq \sup_{v \in V \setminus \{0\}} \frac{b(u_h, \Pi v)}{\|\Pi v\|_V}
= \sup_{v \in V \setminus \{0\}} \frac{b(u_h, v)}{\|\Pi v\|_V}
\geq \frac{1}{C_\Pi} \sup_{v \in V \setminus \{0\}} \frac{b(u_h, v)}{\|v\|_V}.
\]
Combining these estimates yields
\[
\sup_{v_n \in \mathcal{M}_n \setminus \{0\}} \frac{b(u_h, v_n)}{\|v_n\|_V} \geq \frac{\mu}{C_\Pi} \|u_h\|_U,
\]
which establishes the discrete inf-sup condition with constant \( \frac{\mu}{C_\Pi} \).
\end{proof}

\begin{remark}
In contrast to classical finite element settings, the test space \( \mathcal{M}_n \) considered here is represented by a neural network. While this allows for adaptive and expressive residual approximations, verifying conditions such as Fortin stability in practice may rely on empirical testing or theoretical guarantees for neural function approximation~\cite{yarotsky2017error, cybenko1989approximation}.
\end{remark}

\begin{theorem}[Well-Posedness and Stability]
\label{thm:wellposedness}
Let the bilinear form \( b(\cdot,\cdot) \) be continuous and satisfy the inf-sup condition. Then, the continuous problem admits a unique solution \( u \in U \) and
\begin{equation}
\label{eq:stab_cont}
\| u \|_{U} \leq \frac{1}{\mu} \| f \|_{V^*},
\end{equation}
where \( \mu > 0 \) is the inf-sup constant and \( f \in V^* \) is the linear functional representing the right-hand side.

Moreover, if a Fortin operator \eqref{fortin operator} exists, the solution $u_h$ of the discrete problem~\eqref{dis: full dual norm} satisfies


\begin{equation}
\label{eq:stab_cont2}
\| u_h \|_{U} \leq \frac{2}{\mu} \| f \|_{\mathcal{M}^*},
\end{equation}
\end{theorem}

\begin{proof}
From the continuous inf-sup condition:
\[
\inf_{u \in U \setminus \{0\}} \sup_{v \in V \setminus \{0\}} \frac{|b(u, v)|}{\|u\|_U \|v\|_{V}} \geq \mu > 0,
\]
and for the solution \( u \in U \), this implies:
\[
\mu \|u\|_U \leq \sup_{v \in V \setminus \{0\}} \frac{|b(u, v)|}{\|v\|_{V}} = \|f\|_{V^*},
\]
hence the stability estimate \eqref{eq:stab_cont} follows. Existence and uniqueness are guaranteed by the Banach-Ne\v{c}as-Babu\v{s}ka theorem.

Next, we assume that a Fortin operator exists and let \( \tilde{\mu} \) be the discrete inf-sup constant. To bound the discrete norm, we decompose the residual as follows:
\[
\| u_h \|_{U} \leq \frac{1}{\tilde{\mu}} \left( \sup_{v_\theta \in \mathcal{M}_n \setminus \{0\}} \frac{ \langle B u_h - f, v_\theta \rangle }{ \| v_\theta \|_{V
} } + \sup_{v_\theta \in \mathcal{M}_n \setminus \{0\}} \frac{ \langle f, v_\theta \rangle }{ \| v_\theta \|_{V
} } \right).
\]
The first term is minimized by $u_h$ among all $w_h$ in $U_h,$ in particular $w_h = 0,$ in other words:
\[
\sup_{v_\theta \in \mathcal{M}_n \setminus \{0\}} \frac{ \langle Bu_h - f, v_\theta \rangle }{ \| v_\theta \|_{V
} }   \leq  \sup_{v_\theta \in \mathcal{M}_n \setminus \{0\}} \frac{ \langle - f, v_\theta \rangle }{ \| v_\theta \|_{V
} } = \| f \|_{\mathcal{M}_n^*}.
\]
Hence, the resulting bound becomes
\[
\| u_h \|_U \leq \frac{2}{\tilde{\mu}} \| f \|_{\mathcal{M}_n^*}.
\]

\end{proof}
\subsection{Error Estimates}
\label{sec:chapter7:error-estimates}

We now establish a priori and a posteriori error estimates for the hybrid MinRes method. In this framework, the trial space \( U_h \)
is discretized via FEM, and the test space \( \mathcal{M}_n \)
is discretized using a neural network.


\begin{theorem}[A Priori Error Estimate]
Let \( U_h \subset U \) be a subspace of a Hilbert space, and let \( \mathcal{M}_n \subset V \) be a neural-network-parameterized set. Suppose the bilinear form \( b(\cdot, \cdot) \), or equivalently the operator \( B : U \to \mathcal{M}_n^* \), is bounded with norm \( \|B\| \), and satisfies the inf-sup condition with constant \( \mu > 0 \). Assume a Fortin operator (Definition~\ref{def: fortin operator}) exists. Then, for all \( u \in U \) and \( u_h \in U_h \), the following a priori error estimate holds:
\[
\| u - u_h \|_U \leq \left(1 + \frac{2 C_\Pi \| B \|}{\mu} \right) \inf_{w_h \in U_h} \| u - w_h \|_U.
\]
\end{theorem}

\begin{proof}
Let \( w_h \in U_h \) be arbitrary. Then by the triangle inequality:
\[
\| u - u_h \|_U \leq \| u - w_h \|_U + \| w_h - u_h \|_U.
\]
We now bound the second term using the discrete inf-sup condition~\eqref{eq:disc_infsup}:
\[
\| w_h - u_h \|_U \leq \frac{C_\Pi}{\mu} \sup_{v_\theta \in \mathcal{M}_n \setminus \{0\}} \frac{\langle B(w_h - u_h), v_\theta \rangle}{\| v_\theta \|_{\mathcal{M}_n}}.
\]
We add and subtract \( u \) inside the operator:
\[
\langle B(w_h - u_h), v_\theta \rangle = \langle B(w_h - u), v_\theta \rangle + \langle B(u - u_h), v_\theta \rangle.
\]
Taking the supremum and applying the triangle inequality:
\[
\sup_{v_\theta} \frac{\langle B(w_h - u_h), v_\theta \rangle}{\| v_\theta \|} 
\leq \sup_{v_\theta} \frac{\langle B(w_h - u), v_\theta \rangle}{\| v_\theta \|} + \sup_{v_\theta} \frac{\langle B(u - u_h), v_\theta \rangle}{\| v_\theta \|}.
\]
Using that \( u_h \) minimizes the residual \( \| B(u - \cdot) \|_{\mathcal{M}_n^*} \), we get:
\[
\| B(u - u_h) \|_{\mathcal{M}_n^*} \leq \| B(u - w_h) \|_{\mathcal{M}_n^*}.
\]
Hence,
\[
\| w_h - u_h \|_U \leq \frac{2 C_\Pi}{\mu} \| B(u - w_h) \|_{\mathcal{M}_n^*} \leq \frac{2 C_\Pi \| B \|}{\mu} \| u - w_h \|_U.
\]
Combining with the triangle inequality:
\[
\| u - u_h \|_U \leq \| u - w_h \|_U + \frac{2 C_\Pi \| B \|}{\mu} \| u - w_h \|_U = \left(1 + \frac{2 C_\Pi\| B \|}{\mu} \right) \| u - w_h \|_U.
\]
Taking the infimum over \( w_h \in U_h \) completes the proof.
\end{proof}

To support the forthcoming a posteriori estimate, we first establish a preliminary result. The neural-network-parameterized residual approximation \( r_n \in \mathcal{M}_n \) satisfies a quasi-optimality property with respect to the exact residual \( \bar{r} \).

\begin{proposition}[Quasi-Optimal Residual Approximation]
Let \( \bar{r} = R_V^{-1}(f - B u_h) \in V \) be the exact residual, and let \( \mathcal{M}_n \subset V \) be a set of functions parameterized by neural networks, i.e., \( v_\theta \in \mathcal{M}_n \). Define the residual approximation by
\[
r_n = \arg\min_{v_\theta \in \mathcal{M}_n} \left( \frac{1}{2} \|v_\theta\|_V^2 - \langle f - B u_h, v_\theta \rangle \right).
\]
Then \( r_n \) satisfies the quasi-optimality bound:
\[
\| \bar{r} - r_n \|_V \leq \inf_{\theta} \| \bar{r} - v_\theta \|_V.
\]
\end{proposition}

\begin{proof}
Since \( \bar{r} = R_V^{-1}(f - B u_h) \), we have \( \langle f - B u_h, v_\theta \rangle = (\bar{r}, v_\theta)_V \). The objective becomes:
\[
\frac{1}{2} \|v_\theta\|_V^2 - (\bar{r}, v_\theta)_V.
\]
Expanding:
\[
\frac{1}{2} \|v_\theta\|_V^2 - (\bar{r}, v_\theta)_V = \frac{1}{2} \| \bar{r} - v_\theta \|_V^2 - \frac{1}{2} \| \bar{r} \|_V^2.
\]
Since \( \| \bar{r} \|_V^2 \) is constant, minimizing the functional is equivalent to minimizing \( \| \bar{r} - v_\theta \|_V^2 \) over \( \theta \). Therefore, \( r_n \) is the best approximation to \( \bar{r} \) in \( \mathcal{M}_n \), and
\[
\| \bar{r} - r_n \|_V = \inf_{ \theta } \| \bar{r} - v_\theta \|_V.
\]
\end{proof}

This quasi-optimality result allows us to bound the residual approximation error in terms of the best neural-network approximation, which we now use in the a posteriori error analysis.

\begin{theorem}[A Posteriori Error Estimate]
Let \( (u_h, r_n) \in U_h \times \mathcal{M}_n \) be the solution to the mixed minimal-residual problem~\eqref{saddlproblem7}. Then:
\[
\| u - u_h \|_{U} \leq \frac{1}{\mu} \left( \| r_n\|_{V} +  \inf_{\theta} \| \bar{r} - v_\theta\|_{V} \right),
\]
where \( \mu > 0 \) is the inf-sup constant, \( \bar{r} \in V \) is the exact residual.
\end{theorem}

\begin{proof}
From the inf-sup condition:
\[
\|u - u_h\|_{U} \leq \frac{1}{\mu} \|B(u - u_h)\|_{V^*}.
\]
By definition of the Riesz map \( R_{V} : V \to V^* \), we have:
\[
B(u - u_h) = R_{V} \bar{r} \quad \text{therefore,} \quad \|B(u - u_h)\|_{V^*} = \|r\|_{V}.
\]
Let \( r_n \in \mathcal{M}_n \) be the computed residual. Then:
\[
\|\bar{r}\|_{V} \leq \|r_n\|_{V} + \|\bar{r} - r_n\|_{V}.
\]
Now approximate \( \bar{r} \) by some \( v_\theta \in \mathcal{M}_n \). By quasi-optimality:
\[
\| \bar{r} - r_n\|_{V} \leq  \inf_{\theta} \| \bar{r} - v_\theta\|_{V}.
\]
Thus,
\[
\| u - u_h\|_{U} \leq \frac{1}{\mu} \left( \| r_n\|_{V} +  \inf_{\theta} \| \bar{r} - v_\theta\|_{V} \right).
\]
\end{proof}

%% file: sections/section4.tex
\section{Implementation for Weak Advection-Reaction}
\label{sec4:chapter7}

In this section, we describe the practical implementation of the hybrid minimal-residual method. We begin by specifying the model advection-reaction problem, then introduce the shallow ReLU neural network architecture used for approximating the residual space \( \mathcal{M}_n \). Subsequently, we detail the numerical algorithm, the network training procedure, and the formulation of the coupled linear systems solved in each iteration.

\subsection{Model Problem}
\begin{example}\label{adv_reac_example}
Consider the linear advection-reaction equation:
\begin{equation}\label{model:chp7}
    \beta \cdot \nabla u + \gamma u = f \quad \text{in } \Omega, \qquad
    u = u_{\text{in}} \quad \text{on } \Gamma_{\text{in}},
\end{equation}
where \( \Gamma_{\text{in}} \) denotes the inflow boundary, and \( \beta > 0 \), \( \gamma \geq 0 \) are scalar parameters.
\end{example}

We now present a suitable weak formulation for~\eqref{model:chp7} that fits the abstract problem. 
Let \( U := L^2(\Omega) \) and let \(V := H_{0, \Gamma_{\text{out}}}^1(\beta; \Omega),\)
\text{where} 
\[
H_{0, \Gamma_{\text{out}}}^1(\beta; \Omega) := \left\{ w \in L^2(\Omega) \,\middle|\, \beta \cdot \nabla w \in L^2(\Omega),\; w|_{\Gamma_{\text{out}}} = 0 \right\},
\]
equipped with the norm
\[
\|w\|_V^2 = \|w\|_{L^2(\Omega)}^2 + \|\beta \cdot \nabla w\|_{L^2(\Omega)}^2.
\]

The weak formulation of~\eqref{model:chp7} is then: find \( u \in U \) such that
\begin{equation}\label{eq:weak-form2}
b(u, v) = \ell(v) \qquad \forall \, v \in V,
\end{equation}
where for all \( w \in U \), the bilinear and linear forms are defined as:

\begin{align}
b(w, v) &:= \int_{\Omega} w \left( \gamma v - \nabla \cdot (\beta v) \right), \\
\ell(w) &:= \int_{\Omega} f\, w.
\end{align}

\subsection{One-Dimensional Shallow ReLU Neural Network}
\label{sec:shallow-relu}

Following the approach of Liu et al.~\cite{liu2022adaptive}, we adopt a one-dimensional shallow neural network with a single hidden layer of \( n \) neurons and ReLU activation. This defines the function space:
\begin{equation} \label{eq:relu-def}
    \mathcal{M}_n = \left\{ c_0 + \sum_{i=1}^n c_i \, \text{ReLU}(b_i - \omega_i x) : b_i, \omega_i, c_i \in \mathbb{R} \right\},
\end{equation}
where \( c_0 \) and \( c_i \) are output layer weights, and \( b_i \), \( \omega_i \) are hidden layer parameters. The ReLU activation is given by:
\[
    \text{ReLU}(x) = \max\{0, x\} = \begin{cases} 0, & x < 0, \\ x, & x \geq 0. \end{cases}
\]

By fixing \( \omega_i = 1 \) for all \( i \) and setting \( b_i \in (a, b] \), we reduce the number of nonlinear parameters and define:
\begin{equation} \label{eq:simplified-Mn}
    \mathcal{M}_n = \left\{ c_0 + \sum_{i=1}^{n-1} c_i \, \text{ReLU}(b_i - x) : b_i \in (a, b], \, c_i \in \mathbb{R} \right\}.
\end{equation}

This space is equivalent to the classical space of linear splines with free interior knots~\cite{schumaker2007spline}.

Define global basis functions \( \phi_i(x) := \text{ReLU}(b_i - x) \), leading to the representation:
\begin{equation} \label{eq:global-basis-representation}
    v(x) = \sum_{i=0}^n c_i \, \phi_i(x).
\end{equation}

\subsection{Deep Residual Uzawa Algorithm}

The following algorithm is used to compute the approximate solution \( u_h \) and residual \( r_n \).

\begin{algorithm}[H]
\caption{Deep Residual Uzawa Algorithm}
\begin{algorithmic}[1]
\STATE{\textbf{Initialize:} Choose initial guess \( u_h^0 \in U_h \)}
\WHILE{not converged}
    \STATE{\textbf{Residual Minimization:} Solve for \( r_n^k \in \mathcal{M}_n \):
    \begin{equation*}
        r_n^k = \arg\min_{v_n \in \mathcal{M}_n} \left( \frac{1}{2} \|v_n\|^2_{\mathcal{M}} - \langle f - B u_h^k, v_n \rangle \right)
    \end{equation*}}

    \STATE{\textbf{Primal Update:} Solve for \( u_h^{k+1} \in U_h \):
    \begin{equation*}
        \langle u_h^{k+1}, w_h \rangle = \langle u_h^k, w_h \rangle + \rho \, b(w_h, r_n^k), \quad \forall\, w_h \in U_h
    \end{equation*}}

    \STATE{\textbf{Network Optimization:} Update neural network parameters \( \{c_i, b_i\} \) via gradient descent.}

    \STATE{\textbf{Convergence Check:} If \( \|u_h^{k+1} - u_h^k\| < \varepsilon \), terminate.}
\ENDWHILE
\end{algorithmic}
\end{algorithm}

%% file: sections/section5.tex
\section{Numerical Experiments}
\label{sec5:chapter7}

This section presents numerical experiments validating the hybrid minimal-residual approach. We focus on one- and two-dimensional advection-reaction problems with non-smooth or singular source terms. In our experiments of this section, all neural network parameters, denoted by \(\theta = (\{c_i\}, \{b_i\})\), where \(\{c_i\}\) are the outer coefficients and \(\{b_i\}\) are the breakpoints,
were optimized using MATLAB’s \texttt{fminsearchbnd}, an extension of the Nelder-Mead simplex method with box constraints. To avoid breakpoints going outside the domain, we enforce bounds via the \texttt{LB} and \texttt{UB} options. At each Uzawa iteration, the termination tolerance was refined to enhance convergence.

\subsection{One-Dimensional Advection-Reaction Equations}
We now present several one-dimensional test cases to demonstrate the performance of the proposed NN-FEM method. Each case solves the advection-reaction problem defined in Example~\ref{adv_reac_example} with inflow boundary condition \( u_{\mathrm{in}} = 0 \), under varying advection strengths \( \beta \), reaction coefficients \( \gamma \), and source terms. For each experiment, we visualize the FEM approximation \( u_h \) and the corresponding neural residual \( r_n \).
\begin{itemize}
\item \textbf{Case 1:} \(\beta = 0.001\), \(\gamma = 1\), Discontinuous source.

Fig.~\ref{fig:nn_fem_combined_residuals} displays the numerical solutions (left column) and the learned neural residuals (right column). The figure shows the solution and residual plots for \( N = 1 , N = 2\) and \( N = 4 \), with corresponding breakpoints \( M = 4, M =6\) and \( M = 7 \), respectively. 

\item \textbf{Case 2:} \(\beta = 1\), \(\gamma = 0\), Dirac delta source at \( x = 0.5 \).

This case involves a transport-dominated regime with a singular source. The Dirac delta function is approximated numerically. Fig.~\ref{fig:dirac_combined_results1} shows the solution and residual plots for \( N = 1 \) and \( N = 3 \), with corresponding breakpoints \( M = 2 \) and \( M = 5 \), respectively. Despite the singular source, the method produces stable and accurate solutions as the resolution increases.

\item \textbf{Case 3:} \(\beta = 1\), \(\gamma = 0\), Dirac delta source at \( x = \frac{2}{3} \).

In this variation, the singular source is shifted to \( x = \frac{2}{3} \). This further tests the method’s ability to localize sharp features away from the domain center. Fig.~\ref{fig:dirac_combined_results2} presents the solution and residual for \( N = 1 \), \( N = 2 \) and \( N = 4 \). Even with coarse discretisation, the residual adapts well to the shifted source location.

\end{itemize}

\subsection{Two-Dimensional Residual Approximation Test}
\label{subsec:2d-case}

Due to the complexity of implementing a full mixed finite element formulation with neural network residuals in two dimensions, we limit our demonstration to a single illustrative case. The goal is to test whether a neural network can effectively approximate the residual function in a simplified 2D setting.

\begin{example}[2D Advection]
Consider the advection-reaction equation on the unit square \( \Omega = (0,1)^2 \) with constant vertical advection:
\[
\beta\, \partial_y u(x, y) = f(x, y), \qquad u(x, 0) = 0,
\]
where \( \beta = (0,1) \).
\end{example}

To simplify the experiment, we assume the following exact solution:
\[
u(x, y) = \begin{cases}
1, & x > 0.5, \\
0, & x \leq 0.5,
\end{cases}
\]
and a constant finite element approximation:
\[
u_h(x, y) = 0.5.
\]

\begin{figure}[htbp]
\centering
\begin{subfigure}[b]{0.45\textwidth}
    \includegraphics[width=\textwidth]{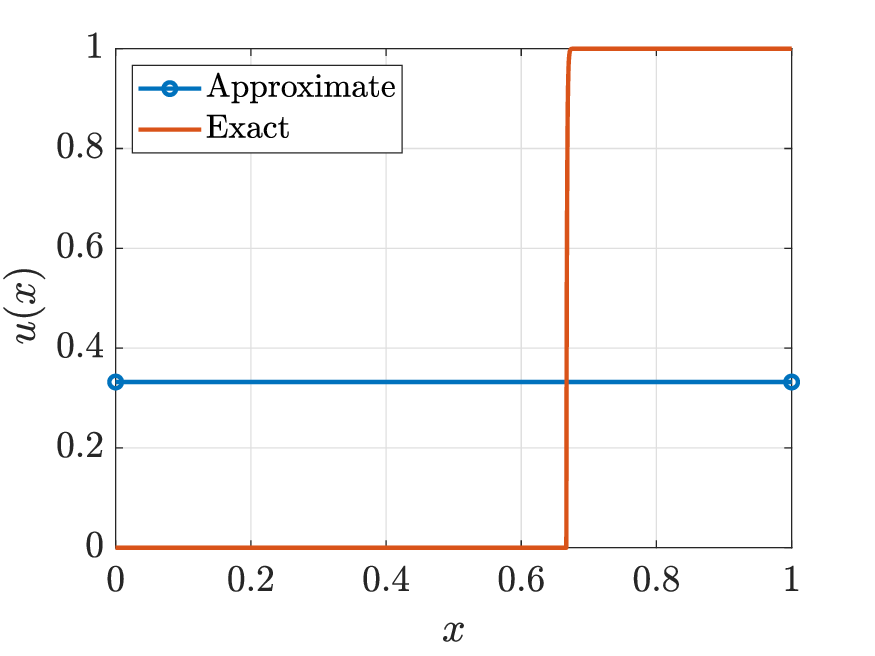}
    \caption{Solution \( u(x) \), mesh size 1.}
\end{subfigure}
\hfill
\begin{subfigure}[b]{0.45\textwidth}
    \includegraphics[width=\textwidth]{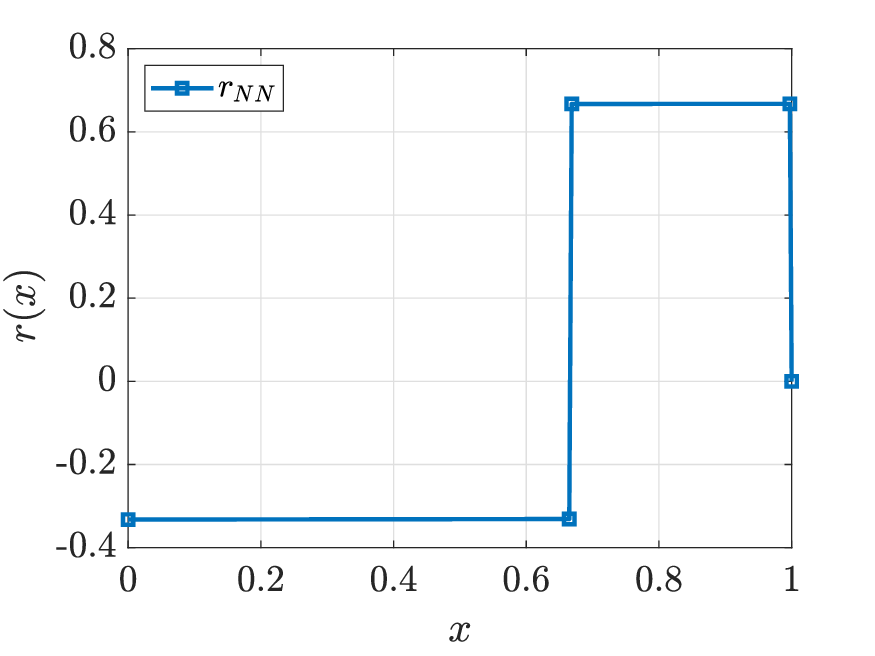}
    \caption{Residual with 4 breakpoints.}
\end{subfigure}


\begin{subfigure}[b]{0.45\textwidth}
    \includegraphics[width=\textwidth]{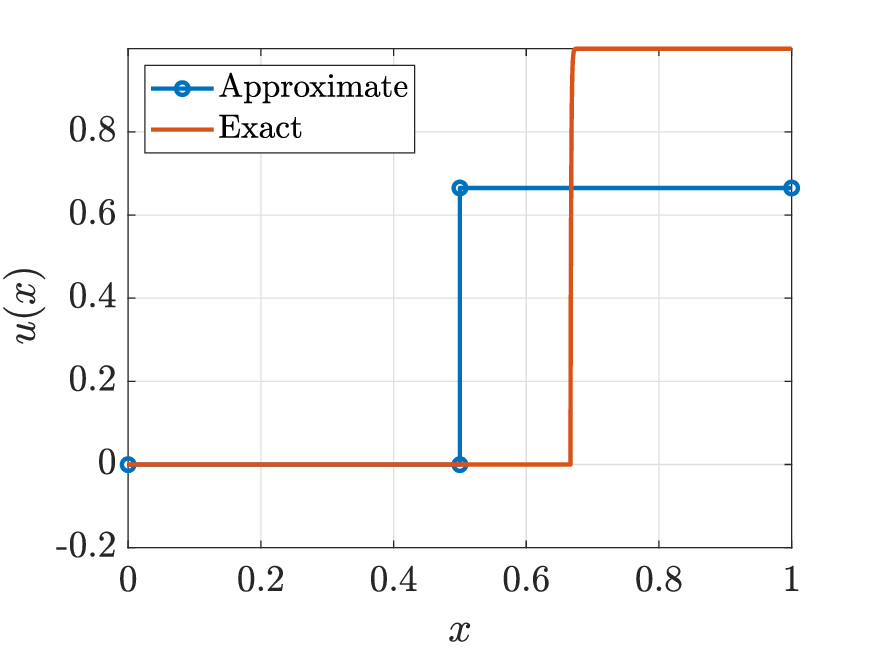}
    \caption{Solution \( u(x) \), mesh size 0.5.}
\end{subfigure}
\hfill
\begin{subfigure}[b]{0.45\textwidth}
    \includegraphics[width=\textwidth]{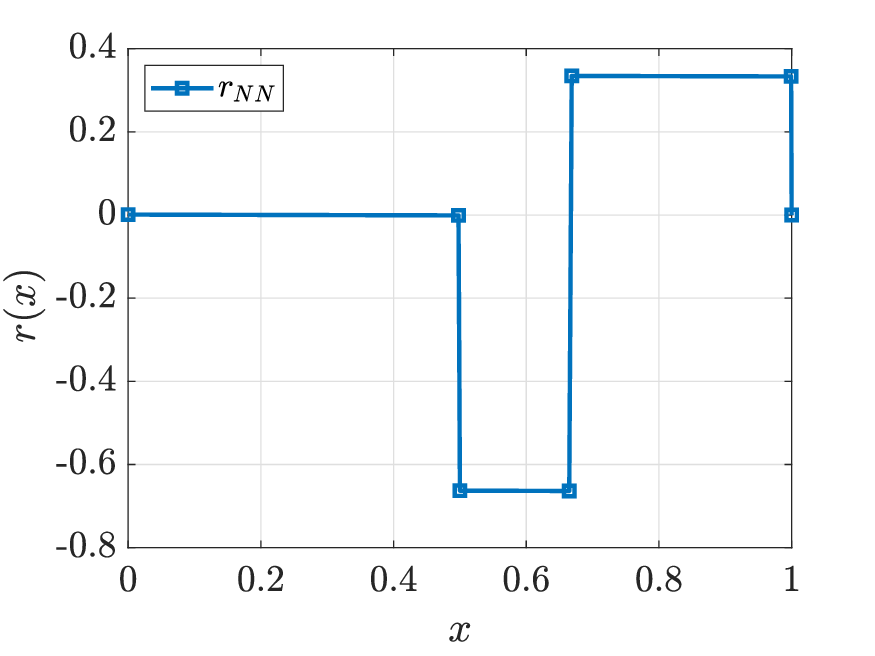}
    \caption{Residual with 6 breakpoints.}
\end{subfigure}


\begin{subfigure}[b]{0.45\textwidth}
    \includegraphics[width=\textwidth]{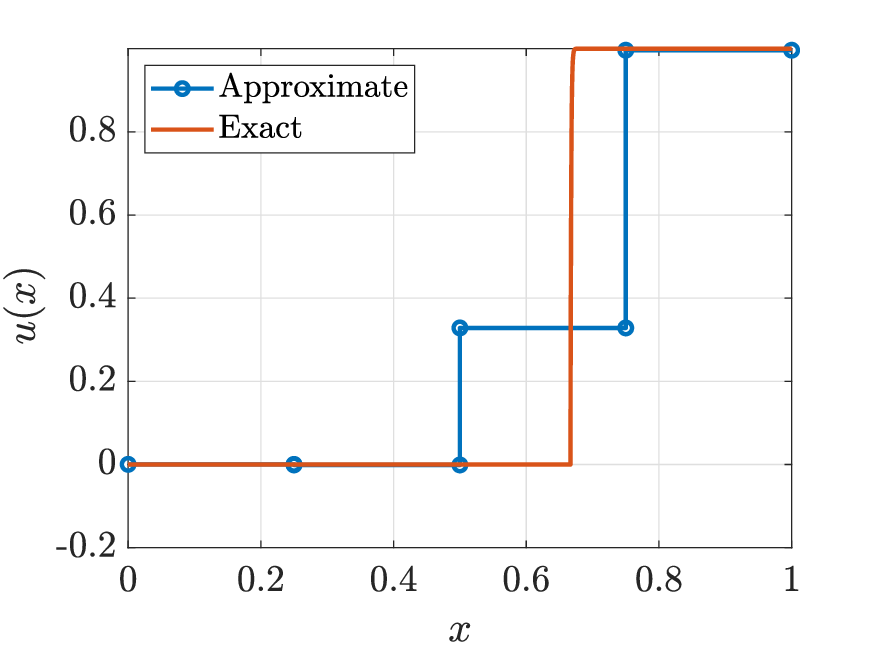}
    \caption{Solution \( u(x) \), mesh size 0.25.}
\end{subfigure}
\hfill
\begin{subfigure}[b]{0.45\textwidth}
    \includegraphics[width=\textwidth]{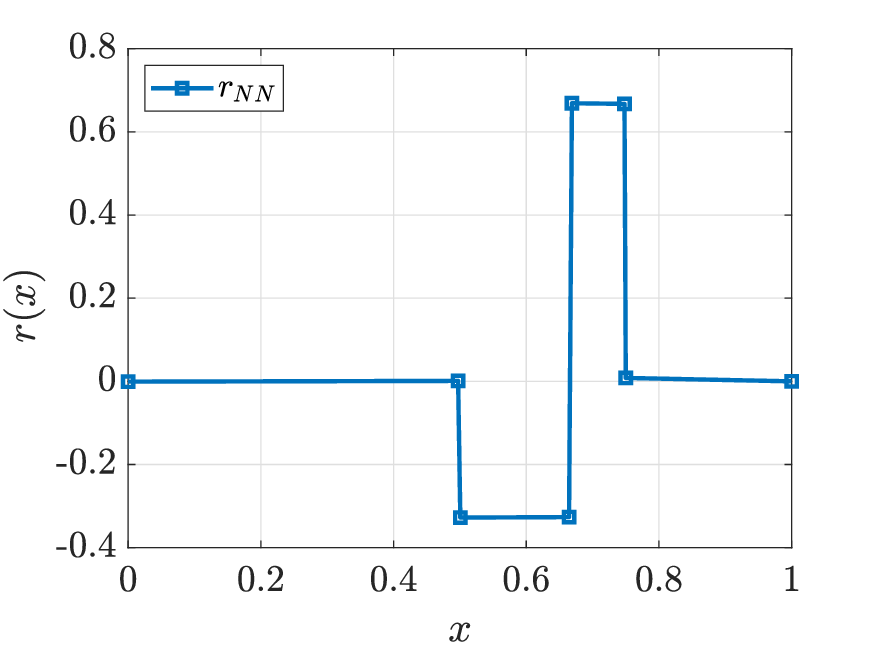}
    \caption{Residual with 7 breakpoints.}
\end{subfigure}

\caption{Left: FEM approximations \( u(x) \) for mesh sizes 1, 0.5, and 0.25. Right: Corresponding neural residuals with 4, 6, and 7 ReLU breakpoints, respectively.}
\label{fig:nn_fem_combined_residuals}
\end{figure}

\begin{figure}[htbp]
\centering
\begin{subfigure}[b]{0.45\textwidth}
    \includegraphics[width=\textwidth]{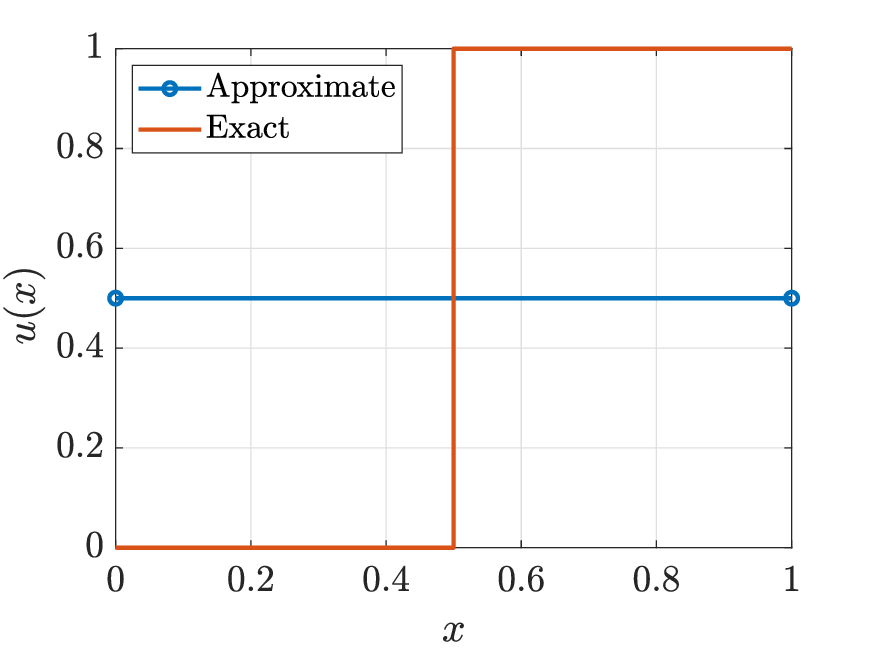}
    \caption{Solution \( u(x) \), mesh size 1.}
    \label{fig:dirac_u_mesh11}
\end{subfigure}
\hfill
\begin{subfigure}[b]{0.45\textwidth}
    \includegraphics[width=\textwidth]{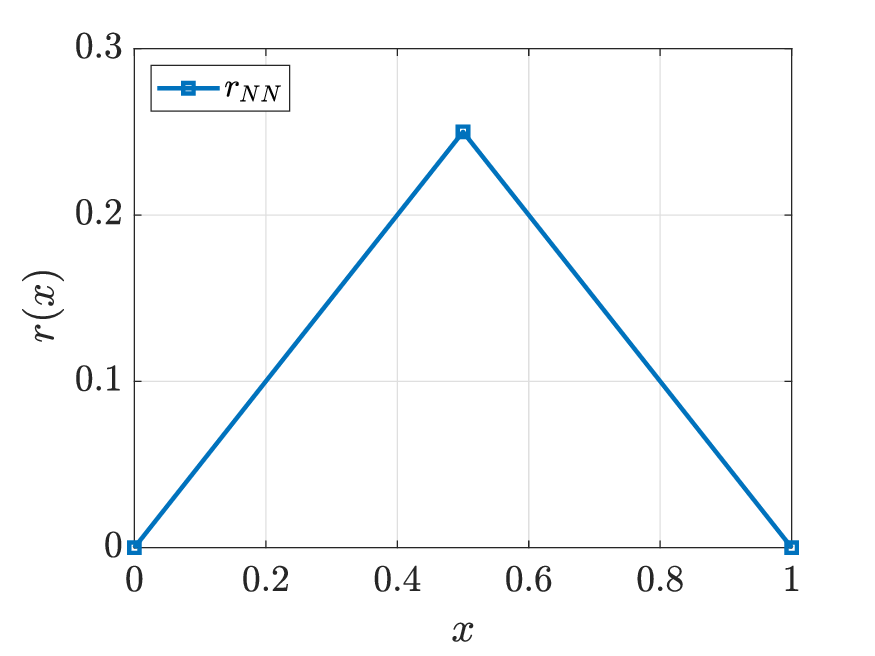}
    \caption{Residual with 2 breakpoints.}
    \label{fig:dirac_r_bp41}
\end{subfigure}


\begin{subfigure}[b]{0.45\textwidth}
    \includegraphics[width=\textwidth]{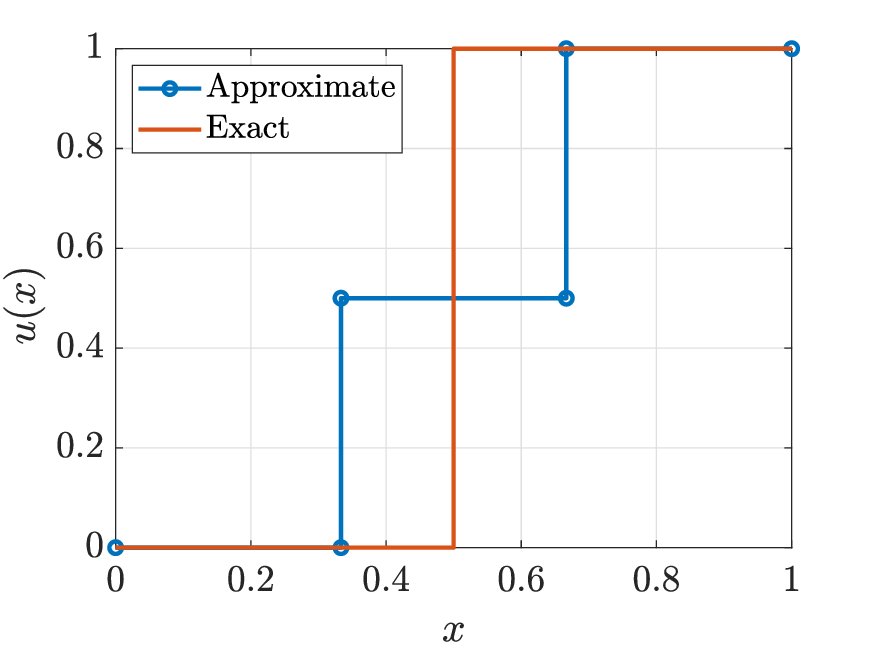}
    \caption{Solution \( u(x) \), mesh size $\frac{1}{3}$.}
    \label{fig:dirac_u_mesh41}
\end{subfigure}
\hfill
\begin{subfigure}[b]{0.45\textwidth}
    \includegraphics[width=\textwidth]{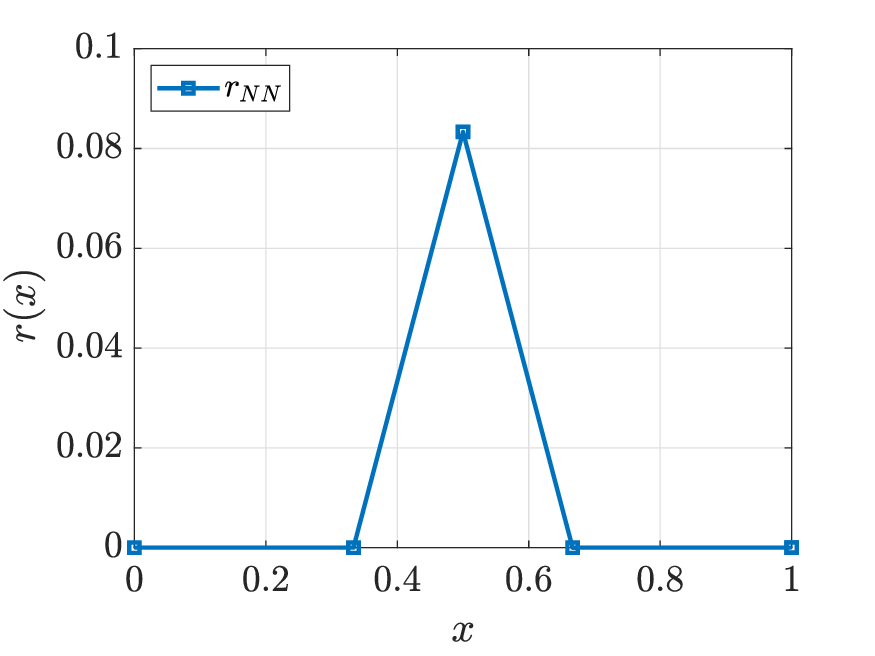}
    \caption{Residual with 5 breakpoints.}
    \label{fig:dirac_r_bp71}
\end{subfigure}

\caption{Left: FEM approximations \( u(x) \) with mesh sizes 1 and $\frac{1}{3}$. Right: Corresponding residuals using 2 and 5 ReLU breakpoints, respectively.}
\label{fig:dirac_combined_results1}
\end{figure}

\begin{figure}[htbp]
\centering
\begin{subfigure}[b]{0.45\textwidth}
    \includegraphics[width=\textwidth]{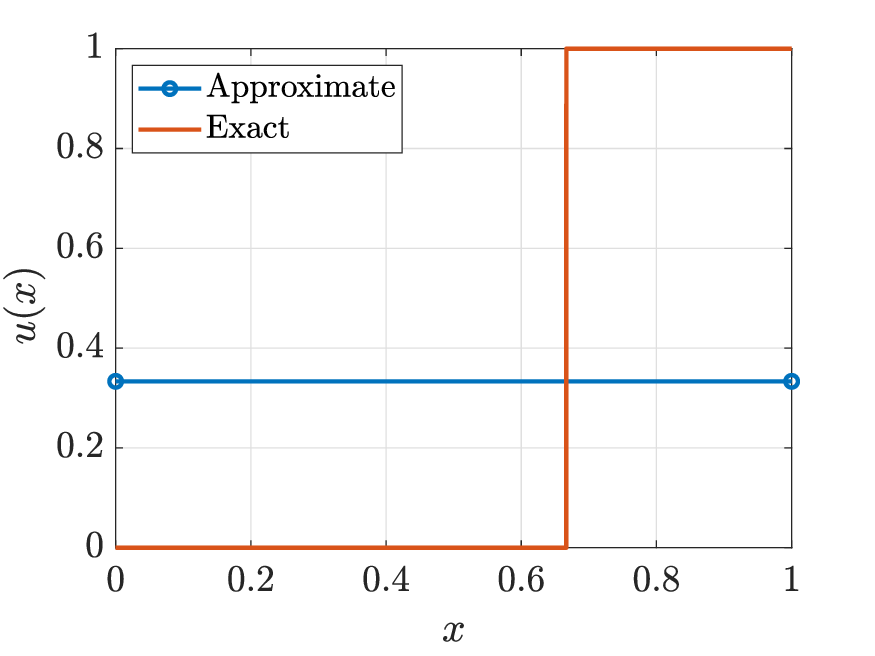}
    \caption{Solution \( u(x) \), mesh size 1.}
    \label{fig:dirac_u_mesh111}
\end{subfigure}
\hfill
\begin{subfigure}[b]{0.45\textwidth}
    \includegraphics[width=\textwidth]{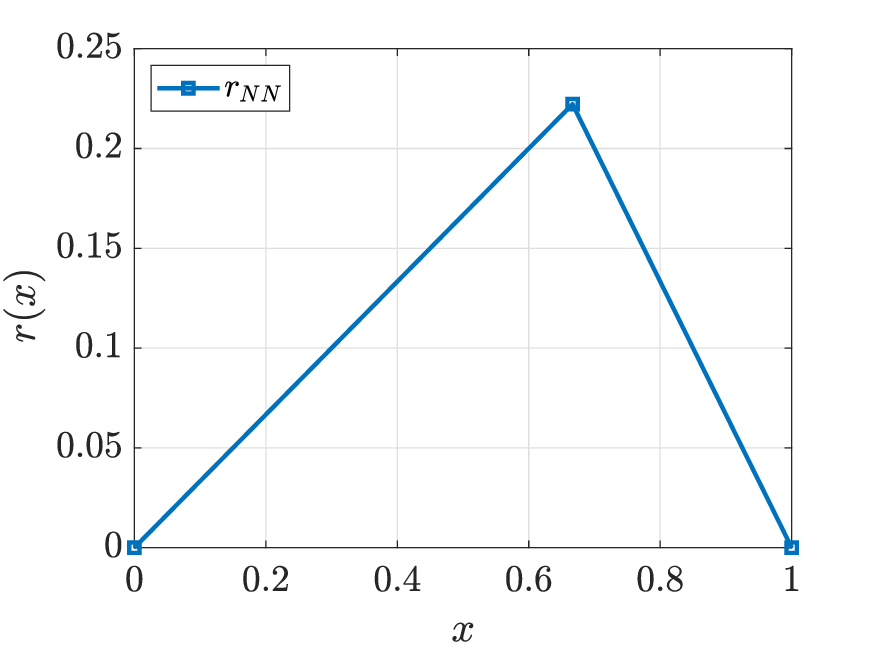}
    \caption{Residual with 2 breakpoints.}
    \label{fig:dirac_r_bp412}
\end{subfigure}


\begin{subfigure}[b]{0.45\textwidth}
    \includegraphics[width=\textwidth]{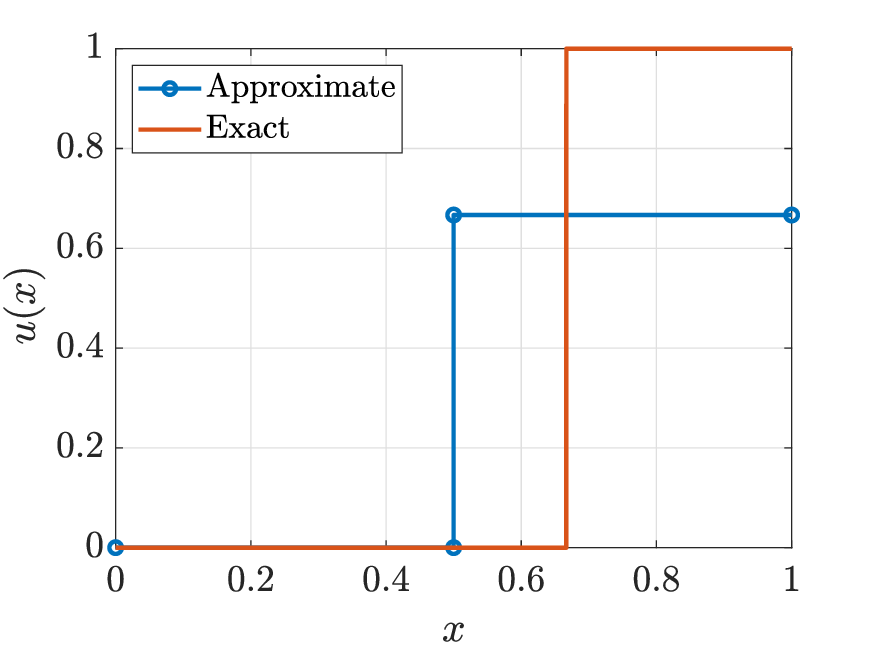}
    \caption{Solution \( u(x) \), mesh size 0.5.}
    \label{fig:dirac_u_mesh413}
\end{subfigure}
\hfill
\begin{subfigure}[b]{0.45\textwidth}
    \includegraphics[width=\textwidth]{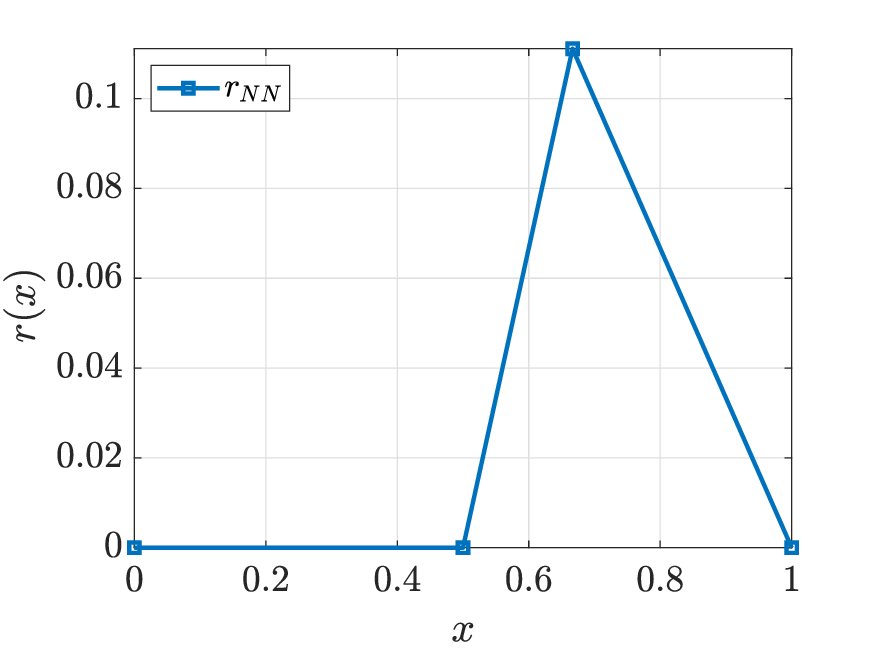}
    \caption{Residual with 3 breakpoints.}
    \label{fig:dirac_r_bp714}
\end{subfigure}


\begin{subfigure}[b]{0.45\textwidth}
    \includegraphics[width=\textwidth]{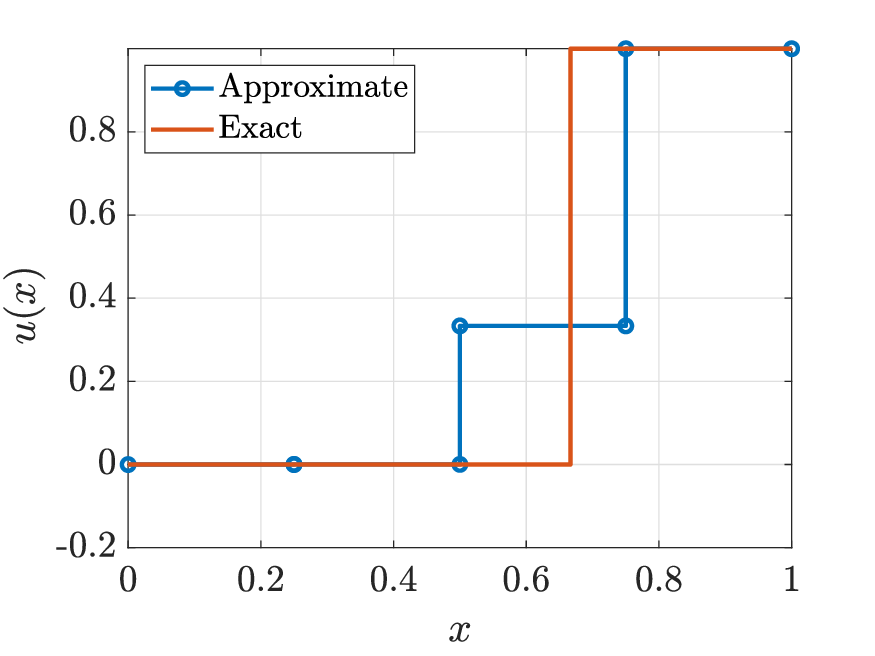}
    \caption{Solution \( u(x) \), mesh size 0.25.}
    \label{fig:dirac_u_mesh415}
\end{subfigure}
\hfill
\begin{subfigure}[b]{0.45\textwidth}
    \includegraphics[width=\textwidth]{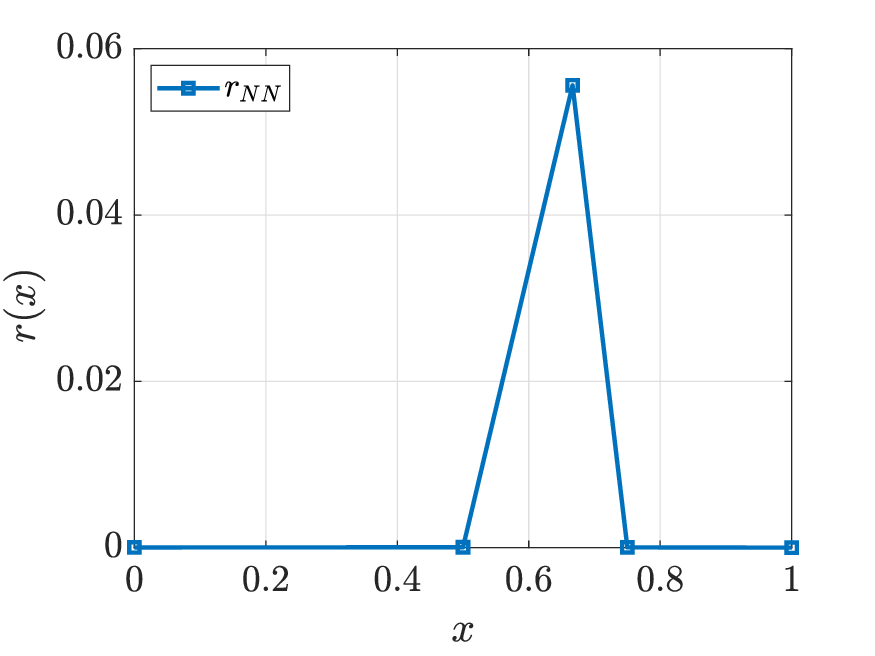}
    \caption{Residual with 4 breakpoints.}
    \label{fig:dirac_r_bp716}
\end{subfigure}

\caption{Solution and residual for Dirac source located at \( x = \frac{2}{3} \). Left: FEM approximations \( u(x) \) with mesh sizes 1, 0.5 and 0.25. Right: Corresponding residuals using 2, 3 and 4 ReLU breakpoints, respectively.}
\label{fig:dirac_combined_results2}
\end{figure}



\begin{figure}[htbp]
\centering
\begin{subfigure}[b]{0.45\textwidth}
    \includegraphics[width=\textwidth]{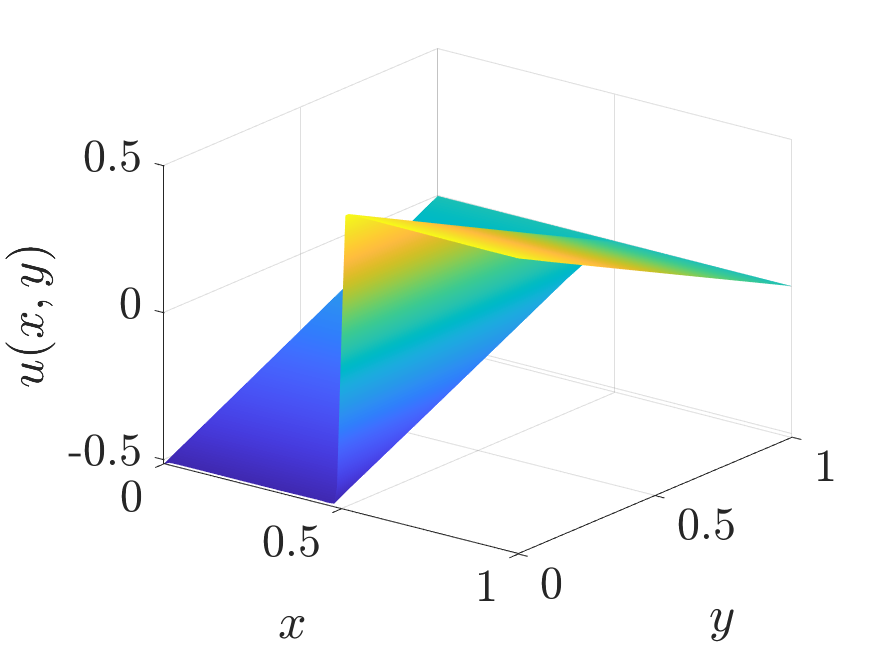}
    \caption{NN-FEM approximation.}
\end{subfigure}
\hfill
\begin{subfigure}[b]{0.45\textwidth}
    \includegraphics[width=\textwidth]{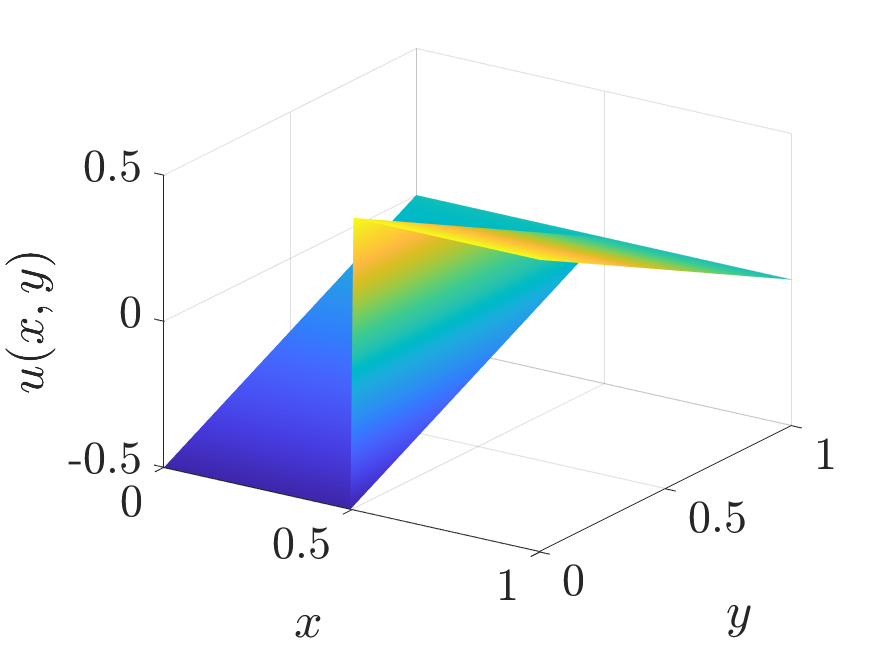}
    \caption{Reference (exact) solution.}
\end{subfigure}
\caption{Residual-informed neural FEM captures structure in 2D advection.}
\label{fig:convergence_figures1}
\end{figure}

We then define the exact residual as:
\[
r(x, y) = f(x, y) - \beta\, \partial_y u_h(x, y) = f(x, y),
\]
since \( u_h \) is constant. For the chosen \( u \), the source term \( f \) is:
\[
f(x, y) = \begin{cases}
0.5(1 - y), & x > 0.5, \\
0.5(y - 1), & x \leq 0.5,
\end{cases}.
\]

We approximate the residual \( r \) using a ReLU neural network trained with a least-squares objective over collocation points in \( \Omega \). The objective functional used is:
\[
\frac{1}{2}\|\beta \cdot \nabla r\|_{L^2(\Omega)}^2 - \int_\Omega u_h\, \beta \cdot \nabla r \, dx + \int_\Omega u\, \beta \cdot \nabla r \, dx,
\]
where \( u_h \) is the finite element approximation and \( u \) is the assumed exact solution.

Figures~\ref{fig:convergence_figures1} demonstrate that, even in two-dimensional settings, the neural-network-based residuals significantly enhance the resolution of directional flow features and sharp gradient zones. This highlights the method’s potential for robust extension to higher-dimensional PDE problems.

\subsection{Error Convergence Plots for NN-FEM}
\label{sec:nn-fem-convergence}

We conclude this section by studying the convergence behavior of the proposed minimal-residual NN-FEM approach. The goal is to quantify the approximation error for both the primal solution \( u_h \in U_h \) and the residual \( r_n \in \mathcal{M}_n \), based on the number of degrees of freedom.

Figures~\ref{fig:u_r_convergence_all} show the convergence behavior under four different advection coefficients \( \beta = 1, 0.1, 0.01, 0.001 \), representing regimes from advection-dominant to reaction-dominant. Moreover, Fig.~\ref{fig:three_valueNwithMs} illustrates the convergence behavior when the source term \( f \) is a Dirac delta function, with parameters \( \beta = 1 \) and \( \gamma = 0 \). These plots reflect convergence with respect to different breakpoints used in the residual network, demonstrating the method's robustness in the presence of singular data.

\begin{figure}[htbp]
    \centering

    \begin{subfigure}[t]{0.48\textwidth}
        \centering
        \includegraphics[width=\textwidth]{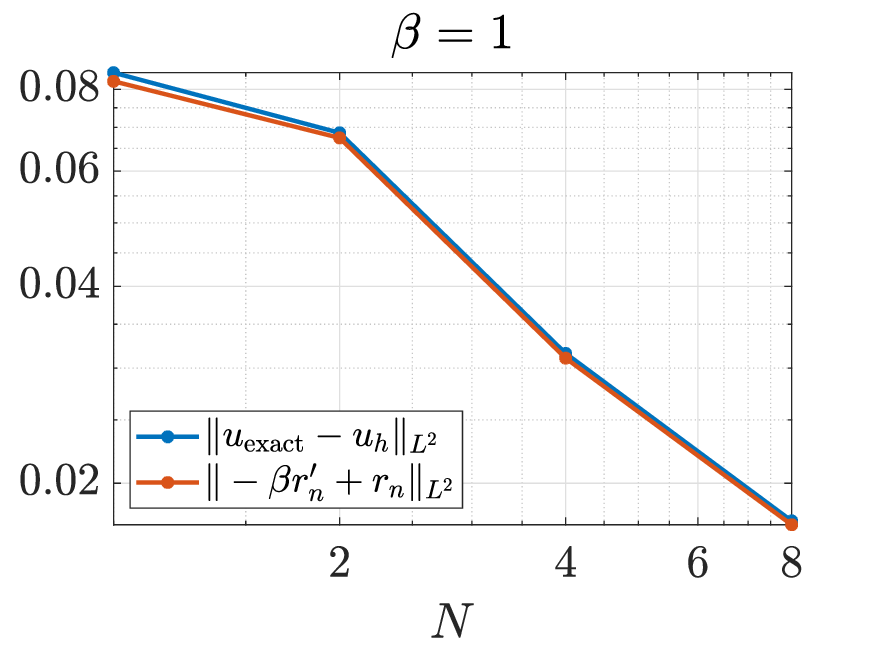}
        \caption{Convergence for \( \beta = 1 \)}
        \label{fig:error_beta1}
    \end{subfigure}
    \hfill
    \begin{subfigure}[t]{0.48\textwidth}
        \centering
        \includegraphics[width=\textwidth]{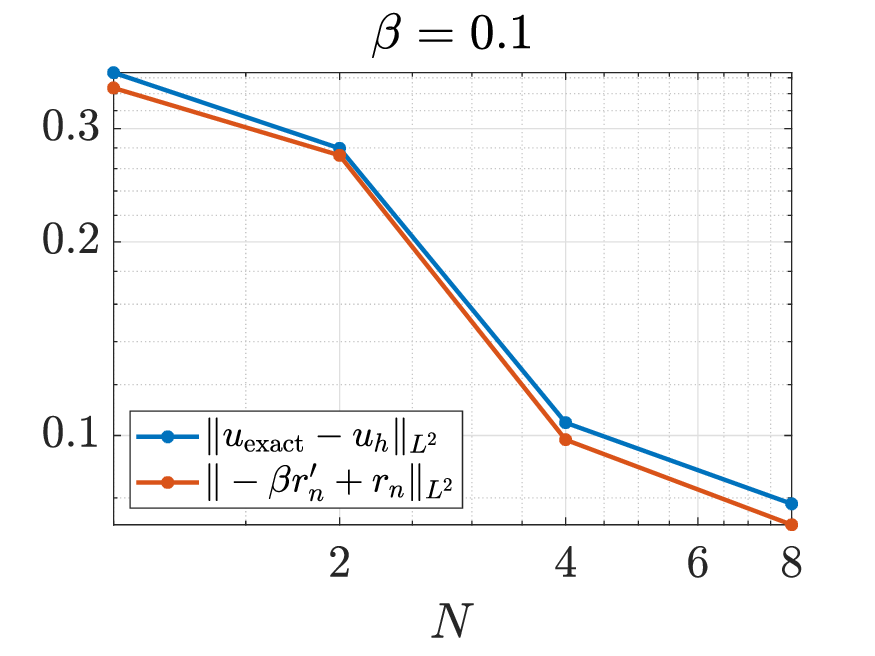}
        \caption{Convergence for \( \beta = 0.1 \)}
        \label{fig:error_beta01}
    \end{subfigure}

    \vskip\baselineskip

    \begin{subfigure}[t]{0.48\textwidth}
        \centering
        \includegraphics[width=\textwidth]{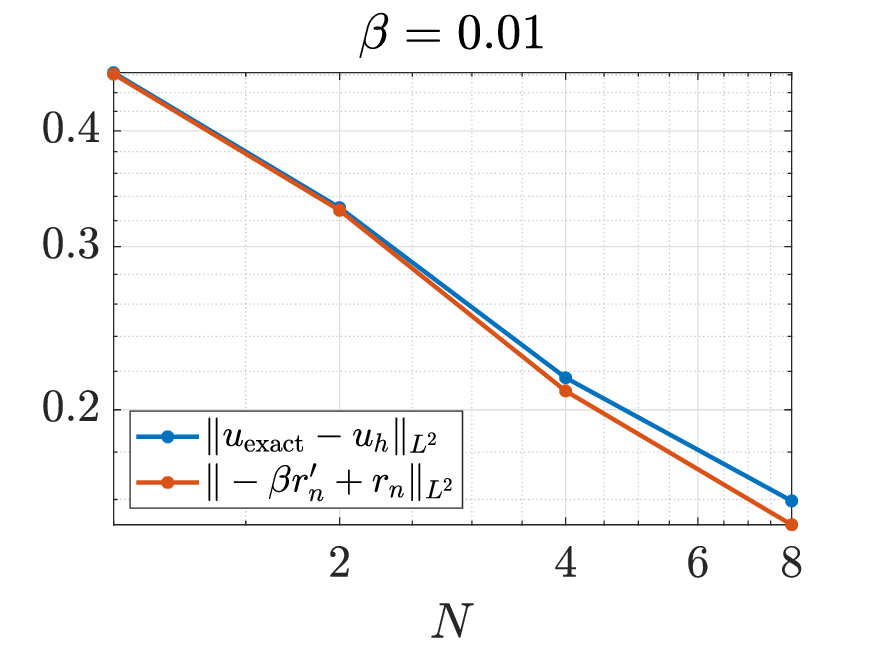}
        \caption{Convergence for \( \beta = 0.01 \)}
        \label{fig:error_beta001}
    \end{subfigure}
    \hfill
    \begin{subfigure}[t]{0.48\textwidth}
        \centering
        \includegraphics[width=\textwidth]{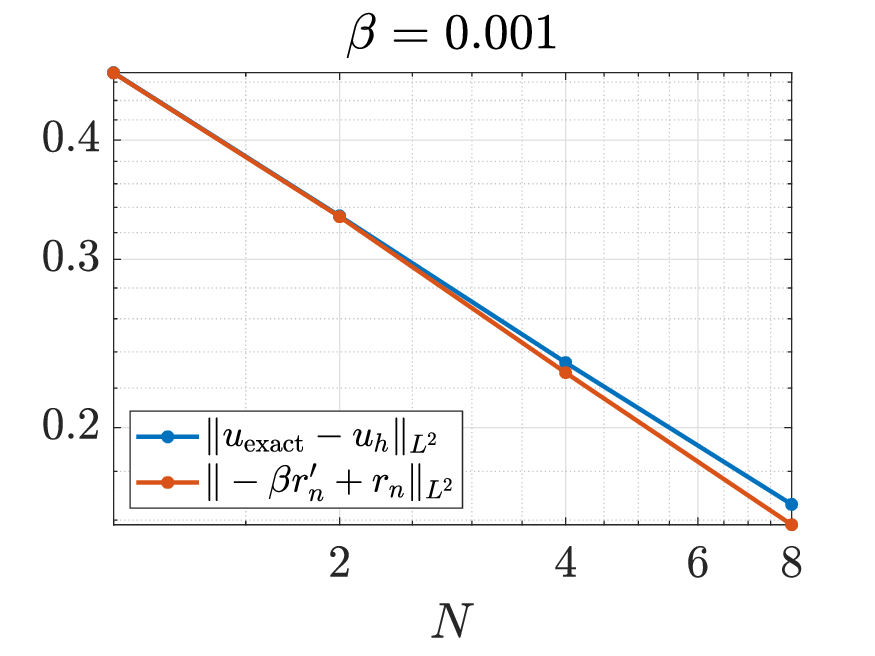}
        \caption{Convergence for \( \beta = 0.001 \)}
        \label{fig:error_beta0001}
    \end{subfigure}

    \caption{Log-log plots of \( L^2 \) errors in the solution \( u_h \) and the residual expression \( -\beta r_n' + \gamma r_n \) for decreasing values of the advection parameter \( \beta \).}
    \label{fig:u_r_convergence_all}
\end{figure}

\begin{figure}[htbp]
    \centering
    \begin{subfigure}[b]{0.32\textwidth}
        \centering
        \includegraphics[width=\textwidth]{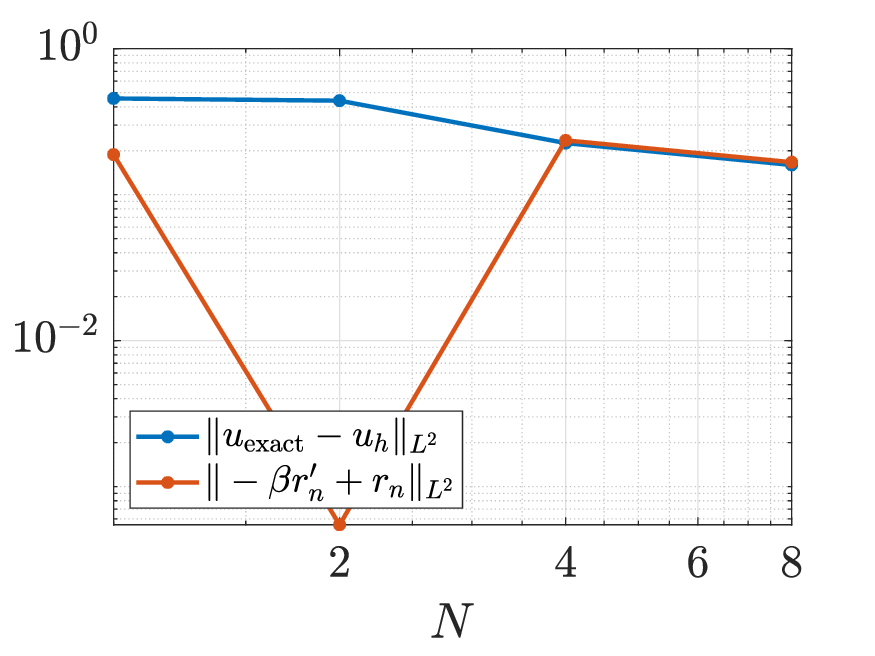}
        \caption{ $M = N $}
        \label{fig:sub1}
    \end{subfigure}
    \hfill
    \begin{subfigure}[b]{0.32\textwidth}
        \centering
        \includegraphics[width=\textwidth]{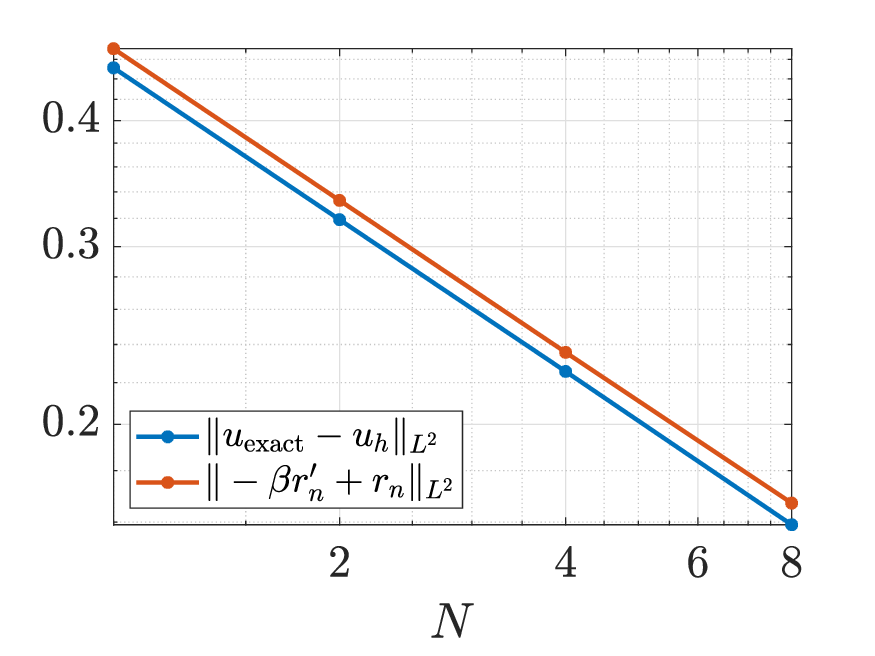}
        \caption{$M = 2 N $}
        \label{fig:sub2}
    \end{subfigure}
    \hfill
    \begin{subfigure}[b]{0.32\textwidth}
        \centering
        \includegraphics[width=\textwidth]{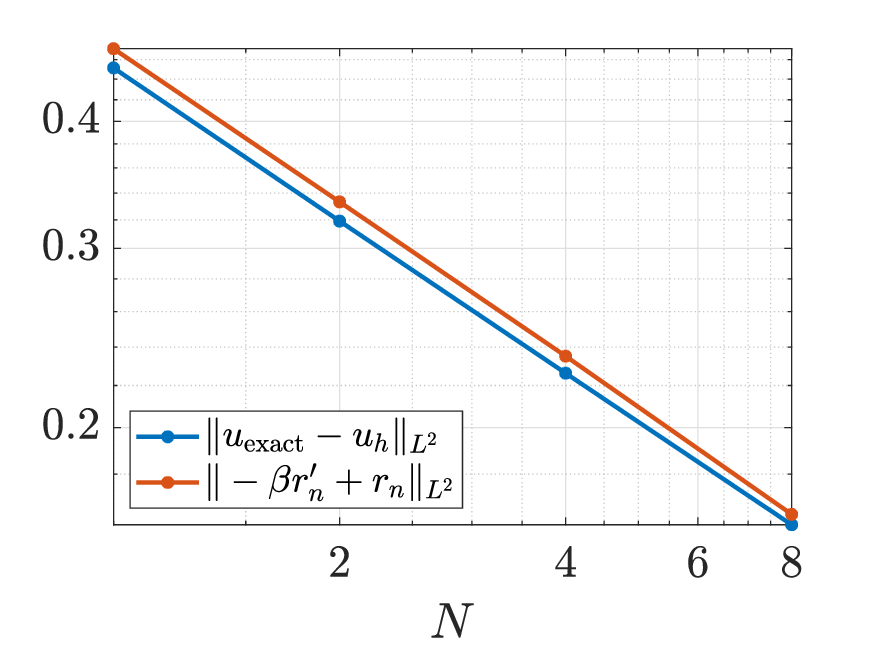}
        \caption{$M = 4 N $}
        \label{fig:sub3}
    \end{subfigure}

    \caption{Convergence plots of $u$ and $r$ for Case 3 under different choices of residual dimension. This figure illustrates $M=N$ (left), $M=2 N$ (middle), and $M=4 N$ (right).}
    \label{fig:three_valueNwithMs}
\end{figure}

%% file: sections/section7.tex
\section{Conclusions}
\label{sec:conclusions}

In this work, we introduced a minimal-residual finite element framework in which neural networks are employed to approximate the test spaces required for the evaluation of residual dual norms. This leads to a deep residual Uzawa algorithm that alternates finite element updates with neural network training, thereby combining the stability of variational formulations with the flexibility of data-driven approximation.

Our analysis established well-posedness of the scheme, discrete inf-sup stability, and a priori as well as a posteriori error estimates. Numerical experiments in one and two dimensions confirmed that the proposed method achieves robust accuracy for advection-reaction problems with singular and discontinuous data, outperforming standard discretizations on coarse meshes.

The results demonstrate that minimal-residual methods enhanced with neural networks can provide a stable and flexible strategy for challenging PDE regimes. 
For future research, several directions for development remain.
From a theoretical perspective, the framework can be extended to nonlinear and higher-order PDEs. Furthermore, progress on constructing or approximating Fortin operators would reinforce its mathematical foundation. Methodologically, the approach can be expanded by parameterizing not only the residual but also the trial solution with neural networks. This would move the framework closer to fully data-driven variational solvers.

Further improvements may come from optimizing the relaxation strategy in the Uzawa iteration. Another promising direction is to explore alternative solvers, such as multigrid, Krylov subspace, or preconditioned descent methods, to accelerate convergence. Together, these directions open the path to more powerful and scalable neural minimal-residual formulations for complex multi-physics problems.